\documentclass[11pt]{amsart}
\usepackage{amsfonts, amssymb, amsmath, amsthm}
\usepackage{amsfonts}
\usepackage[dvips]{graphics}
\usepackage{epsfig}
\usepackage{graphicx}
\usepackage{psfrag}
\usepackage{amsmath}
\usepackage{amssymb,subfigure}
\usepackage[colorlinks=true, pdfstartview=FitV, linkcolor=blue,
            citecolor=blue, urlcolor=blue]{hyperref}
\usepackage{pifont}

\usepackage{epstopdf}

\usepackage[usenames,dvipsnames]{xcolor}

\numberwithin{equation}{section} \oddsidemargin=-.0cm
\numberwithin{equation}{section} \oddsidemargin=-.0cm
\usepackage{newtxtext}

\usepackage{amscd}
\usepackage{amsmath}
\usepackage{amssymb}
\usepackage{graphicx}%
\def\eps{\varepsilon}

\def \and{\quad\text{and}\quad}

\def\epsilon{\varepsilon}

\newcommand{\be}{\begin{equation}}
\newcommand{\ee}{\end{equation}}
\newcommand{\bes}{\begin{equation*}}
\newcommand{\ees}{\end{equation*}}
\newcommand{\baa}{\begin{array}}
\newcommand{\eaa}{\end{array}}
\newcommand{\ba}{\begin{eqnarray}}
\newcommand{\ea}{\end{eqnarray}}

\newcommand{\bea}{\begin{equation} \begin{aligned}}
\newcommand{\eea}{\end{aligned} \end{equation}}

\newcommand{\mk}{\color{black}}

\numberwithin{equation}{section}

\definecolor{rred}{rgb}{0.7,0,0.1}
\definecolor{greenrb}{rgb}{0.2,0.6,0.2}

\newcommand{\mkrevb}{\color{black}}

\newcommand{\mkr}{\color{black}}
\newcommand{\mkrev}{\color{black}}

\newtheorem{Lem}{Lemma}[section]

\newtheorem{Cor}{Corollary}[section]
\newtheorem{Ex}{Example}[section]
\newtheorem{Def}{Definition}[section]
\newtheorem{Rmk}{Remark}[section]

\definecolor{rred}{rgb}{0.7,0,0.1}
\newcommand{\mkk}{\color{black}}

\def \au {\rm}

\def \bk {\it}

\def\c{\mathfrak{c}}
\def\s{\mathfrak{s}}

\def \no#1#2#3 {{\bf #1} (#3), #2.}
\def \eds#1#2#3 {#1, #2, #3.}
\newtheorem{proposition}{Proposition}[section]
\newtheorem{theorem}[proposition]{Theorem}

\newtheorem{lemma}[proposition]{Lemma}
\theoremstyle{definition}

\numberwithin{equation}{section}
\title[Topological instabilities in families  of semilinear parabolic problems] {Topological instabilities in families  of semilinear parabolic problems subject to nonlinear perturbations}

\author[Micka\"el D. Chekroun]{Micka\"el D. Chekroun}

\address[MDC]{Department of Atmospheric \& Oceanic Sciences, University of California, Los Angeles, CA 90095-1565, USA} 
\email{mchekroun@atmos.ucla.edu}

\date{}

\thanks{}


\begin{document}
\maketitle
%\vskip-5mm

%

\begin{abstract}
%It is known that for a broad class of one-parameter families of semilinear elliptic problems, a discontinuity in the branch of minimal solutions can be induced by arbitrarily small perturbations  of the nonlinear term, when the spatial dimension is either equal to one or two.
%In this article, we communicate, from a topological point of view, on the dynamical implications of such a sensitivity of the minimal branch, for the corresponding one-parameter  family of semilinear parabolic problems.   

%The key ingredients to do so, rely on a combination of a general continuation result from the Leray-Schauder degree theory regarding the existence of an unbounded continuum of solutions to one-parameter families of elliptic problems,  and a growth property of the branch of minimal solutions to such problems. 

In this article it is proved that the dynamical properties of a broad  class
of semilinear parabolic problems are sensitive to arbitrarily
small but smooth perturbations of the nonlinear term, when the spatial dimension is either equal to one or two. 
This topological instability is shown to result from a local deformation of  the global bifurcation diagram  associated with the corresponding elliptic problems. Such a deformation is shown to systematically occur via the creation of either a multiple-point or a new fold-point on  this diagram when an appropriate small perturbation is applied to the nonlinear term. 

More precisely, it is shown that for a broad class of nonlinear elliptic problems, one can always find an arbitrary small perturbation of the nonlinear term, that generates  (for instance) a local S on the bifurcation diagram whereas the latter is e.g.~monotone when no perturbation is applied;  substituting thus a single solution by several ones. Such an increase in the local multiplicity of the solutions to the elliptic problem results then into a topological instability for the corresponding parabolic problem.

The rigorous proof of the latter instability result requires though to revisit the classical concept of topological equivalence to encompass important cases for the applications such as semi-linear parabolic problems  for which the semigroup may exhibit non-global dissipative properties, allowing for the coexistence of blow-up regions and local attractors in the phase space; cases that arise e.g.~in combustion theory.

A revised framework of topological robustness is thus introduced in that respect within which 
the main topological instability result is then proved for continuous, locally Lipschitz but not necessarily $C^1$ nonlinear terms, that prevent in particular  the use of linearization techniques,  and for which  the family of semigroups may exhibit non-dissipative properties.

\end{abstract}

\tableofcontents

\section{Introduction}
The bifurcations occurring in semilinear or elliptic parabolic problems have been thoroughly
investigated since the pioneering  works of
\cite{Amann,rab1,Crandall_Rabinowitz'73,marsden1976hopf, Hen81,Sattinger_book,sattinger1980bifurcation}, among others. A large portion of the subsequent works  has been devoted to the study of qualitative changes occurring within
a fixed family of such problems when a bifurcation parameter is varied; see e.g.~\cite{MW05,MW14,HI11,kielhofer2012} and references therein. 

Complementarily,  {\it perturbed bifurcation problems}  arising in families of semilinear elliptic equations, have been considered. These problems, in their general formulation, are concerned with the dependence of the global bifurcation diagram to  perturbations of the nonlinear term \cite{keener1973perturbed}. Such a dependence problem is of fundamental importance 
to understand, for instance, how the multiplicity of solutions of such equations varies as the nonlinearity is subject to small disturbances, or is modified due to model imperfections \cite{brezzi1982numerical,golubitsky1979theory,keener1973perturbed}.  

However, this problem has been mainly addressed in the context of two-parameter families  of elliptic problems; see e.g.~\cite{BCT88,BCT88_b,brezzi1980,brezzi1982numerical,damil1990new,Chow_al75,Du, Du_Lou, keener1973perturbed,Korman_Li,Lions,moore1980,She80, spence1982non}. In comparison, the dependence of the global bifurcation diagram with respect to variations in other degrees of freedom such as the ``shape" of the  nonlinearity remains largely unexplored;  
see however \cite{Dancer_domain_effect,Dancer_domain_effect2, Hen05,Nagasaki} for {\mkrevb a study of effects related to the domain's variation.}

As we will see, the study of {\it perturbed} bifurcation problems of semilinear elliptic equations can be naturally related to the study of a notion of  
 {\it topological robustness} of dynamical properties associated with the corresponding  families of semilinear parabolic equations, once the appropriate framework has been set up. The issue is here not only to translate the deformations of the global bifurcation diagram of the elliptic problems  into a dynamical language for the parabolic problems, but also to take into consideration the possible discrepancies  of regularity that may arise between  the weak solutions of the former and the semigroup equilibria of the latter.   
 
  It is the purpose of this article to introduce such a framework that  allows  us in particular, to analyze from a topological viewpoint, the perturbation effects  of the nonlinear term on the  parameterized families of  semigroups associated with semilinear parabolic problems  of the form
 \bea\label{Eq_parab_intro}
\partial_t u-\Delta u & =\lambda g(u), \; \mbox{ in } \Omega,\\
  u & =0, \quad\quad \mbox{ on } \partial\Omega,\;
\eea
given on a bounded and sufficiently smooth domain $\Omega \subset \mathbb{R}^d$. Our approach allows us to include both dissipative\footnote{In the sense that the associated semigroup exhibits a bounded absorbing set; see \cite{Temam3}.} {\mkrevb as well non-dissipative cases with finitely  many local attractors}; the latter cases being commonly encountered when $g$ is {\it superlinear} such as in gas combustion theory \cite{Bebernes_Eberly,F,Gelfand_63,quittner2007superlinear} or in plasma physics \cite{Brezis_Beres,ch,Tem75}, see also \cite{Fil05}.

Within this framework, it is then proved that the dynamical properties of a broad  class
of semilinear parabolic problems turns out to be sensitive to arbitrarily
small perturbations of the nonlinear term, {\mkk when the spatial dimension $d$ is either equal to one or two.}

This is essentially  the content of Theorem \ref{THM_Main} proved below and which constitutes the main result of this article.  
{\mkrevb The proof of this theorem is articulated around a combination of  techniques relative to (i) the {\it generation of discontinuities} in the minimal branch obtained from the perturbative approach of  \cite{Nao}; (ii) 
the {\it growth property} of the branch of minimal solutions (see Proposition \ref{Prop_classical-results} below);  and (iii) a general {\it continuation result} from the Leray-Schauder degree theory, formulated as Theorem \ref{THM_global_unbounded} below.  
The latter theorem provides conditions of existence of an unbounded continuum of steady states for the corresponding family of semilinear elliptic problems.\footnote{Considered in 
$(0,\infty)\times E$, where $E$ is a Banach space for which the nonlinear elliptic problem $-\Delta u=\lambda g(u)$,  $u\vert_{\partial \Omega}=0$, is well-posed, for $\lambda \in \Lambda \subset (0,\infty)$.}}

The proof of Theorem \ref{THM_Main} provides furthermore the mechanism at the origin of  the aforementioned topological instability of the parameterized family of ``phase portraits'' associated with \eqref{Eq_parab_intro}. More precisely, {\mkr it is shown that such a topological instability  comes} from a local deformation of  the $\lambda$-bifurcation diagram  associated with the corresponding elliptic problems.

This deformation is the consequence of the {\mkrevb creation of either a multiple-point or a new fold-point on this diagram} when an appropriate small perturbation is applied to the nonlinear term. This topological signature is proved  for locally Lipschitz but not necessarily $C^1$ nonlinear terms, that prevent in particular the use of 
linearization techniques.   {\mkrevb Furthermore, as will be explained, the results apply to family of semigroups associated with \eqref{Eq_parab_intro}  that may exhibit non-global dissipative properties with coexistence of blow-up regions and finitely many local attractors.}

{\mkrevb Throughout this article, we have tried to make the expository as much self-contained as possible. In that respect, a very brief introduction to the standard notion of  structural stability for dissipative semilinear parabolic equations is provided in Section \ref{Sec_struc_stab}, preceded by a short presentation of the perturbed Gelfand problem in Section \ref{Sec_perturbed-Gelfand} to motivate, in part, the type of problems considered hereafter.  The core of this article is then articulated around Section \ref{Sec_top_rob} and Section \ref{Sec_Main}.  

Section \ref{Sec_top_rob} introduces an abstract framework for the description of topological equivalence between families of semilinear parabolic equations  which may exhibit for instance a mixture of trajectories that blow up or are attracted by equilibria, depending on the ``energy'' contained in the initial data.  In particular, this framework allows us to take into account the possible discrepancies  of regularity that may arise between  the weak solutions of the corresponding elliptic problems and the semigroup equilibria. Section \ref{Sec_Main} presents then the main abstract result  of this article (Theorem \ref{THM_Main}) that is applied on the  perturbed Gelfand problem of Section \ref{Sec_perturbed-Gelfand} as an illustration (Corollary \ref{Main_cor}). Numerical results are then provided in Section \ref{Sec_num}. Finally, Appendix \ref{Sec_AppendixA}  provides a  proof of the continuation result (Theorem \ref{THM_global_unbounded}) used in the proof of Theorem \ref{THM_Main}. }

\section{A revised framework for the topological robustness of families of  semilinear parabolic equations}\label{Sec_topo_rob_mother}

In Section \ref{Sec_perturbed-Gelfand} that follows, the  perturbed Gelfand problem serves as an illustration of {\it perturbed bifurcation problems} arising in families of semilinear elliptic equations. These problems are concerned with  
the dependence of the global bifurcation diagram
to perturbations of the nonlinear term \cite{keener1973perturbed}. As mentioned in Introduction,  such a dependence problem is of fundamental importance 
to understand, for instance, how the multiplicity of solutions of such equations varies as the nonlinearity is subject to small disturbances, or is modified due to model imperfections \cite{brezzi1982numerical,golubitsky1979theory,keener1973perturbed}.

 We will illustrate in Section \ref{Sec_top_rob} below, how  perturbed bifurcation problems can be naturally related to  the study of a certain notion of topological robustness of  
  the corresponding  families of semilinear parabolic equations. Although related to the more standard notion of structural stability encountered for dissipative semilinear parabolic problems \cite{Hale_al'84} (see Section \ref{Sec_struc_stab} below), our  notion of topological robustness is more flexible. As discussed hereafter, our approach,  based on the notion of topological equivalence between parameterized families of semigroups  such as introduced in ~Definition \ref{DEF_Equiv_Rel-family} below (see Section \ref{Sec_top_rob}), adopts indeed a more global viewpoint and allows us  to deal with semigroups not necessarily restricted to an invariant set and associated with parabolic problems in which  a mixed behavior can occur.\footnote{Such  semigroups are typically defined on the set of bounded trajectories, disregarding  the trajectories that undergo a finite-time blow-up or that are defined for all time but are not bounded, the so-called grow-up solutions (see e.g.~\cite{Ben_Gal}).} Furthermore, our approach allows us to take into account  the possible  discrepancies of regularity that may arise between  the (weak) solutions of elliptic problems, on the one hand, and the semigroup equilibria  of the corresponding parabolic problems, on the other.

\subsection{The perturbed Gelfand problem as a motivation}\label{Sec_perturbed-Gelfand}

Given a smooth bounded domain $\Omega\subset\mathbb{R}^{d}$, the
perturbed Gelfand problem,  consists of solving the following nonlinear eigenvalue problem
\begin{equation}\label{Eq_Gelfand-perturb}
\left\{
\begin{array}{l}
\displaystyle -\Delta u = \lambda \exp\Big(\frac{u}{1+\epsilon u}\Big), \;\textrm{ in }\Omega,\\
\hspace{1.8em} u =0,  \qquad \qquad \qquad \quad  \textrm{ on }\partial\Omega,
\end{array}\right.
\end{equation}
of unknowns $\lambda\geq 0$  and $u$ in some functional space. We refer to   \cite{Bebernes_Eberly, ch, F, Gelfand_63,Joseph_Lund, Taira_book, Taira,quittner2007superlinear} for more details regarding the physical contexts where such a problem arises.

{\mkrevb We first recall some general features regarding  the structure of the $\lambda$-parameterized solution set  of \eqref{Eq_Gelfand-perturb}.  These features  can be derived by 
application of topological degree arguments (see Theorem \ref{THM_global_unbounded}) and the theory of semilinear  elliptic equations \cite{Nao_book}.
In the same time, we point out some open questions related to the exact shape of this solution set when the nonlinearity is varied by changing  $\epsilon$.

The goal is here to illustrate on this example the difficulty of characterizing the qualitative changes occurring in the  $\lambda$-bifurcation diagram, when a perturbation, monitored here by $\epsilon$, is applied to the nonlinearity.  As shown in Section \ref{Sec_Main} below, Theorem \ref{THM_Main} allows for a clarification of such qualitative changes for a broad class of nonlinearities subject to arbitrarily small perturbations with compact support.\footnote{Although the allowable perturbations by Theorem \ref{THM_Main} do not include those associated with a variation of $\epsilon$ on this particular example, sensitivity results can still be derived for \eqref{Eq_Gelfand-perturb} by application of Theorem \ref{THM_Main}; see Corollary \ref{Main_cor} below. We refer also to Section \ref{Sec_num} for numerical results when \eqref{Eq_Gelfand-perturb} is subject to perturbations not compactly supported.} }

Let $\alpha \in(0,1)$ and let us consider the H\"older spaces
$V=C^{2,\alpha}(\overline{\Omega})$ and
$E=C^{0,\alpha}(\overline{\Omega}).$ It is well known (see {\it
e.g.} ~\cite[Chapter 6]{Gilbarg_Trudinger}) that given $f\in E$ and $\lambda \geq 0$,
there exists a unique $u\in V$ of the following Poisson problem,
\begin{equation}\label{Es-sol-map}
\left\{
\begin{array}{l}
\displaystyle -\Delta u = \lambda \exp\Big(\frac{f}{1+\epsilon f}\Big), \textrm{ in }\Omega,\\
\hspace{1.8em} u=0,  \qquad \qquad \qquad \quad  \textrm{ on }\partial\Omega.
\end{array}\right.
\end{equation}

One can thus define a solution map $S:E\rightarrow V$ given by $S(f)=u,$
where $u\in V$ is the unique solution to \eqref{Es-sol-map}. By
composing $S$ with the compact embedding $i:V\rightarrow E$
\cite{Gilbarg_Trudinger} we obtain then a map $\widetilde{S}:=i\circ S:E
\rightarrow E$ which is completely continuous.

Define now $G : \mathbb{R}^+\times E \rightarrow E$ by
$G(\lambda,u)=\lambda \widetilde{S}(u),$ and consider the equation,
\begin{equation}\label{Eq-Eq-bif-Gelfand}
\mathcal{G}(\lambda,u):=u-G(\lambda,u)=0_E.
\end{equation}
The mapping $\mathcal{G}$ is a completely continuous perturbation of the
identity and solutions of the equation $\mathcal{G}(\lambda,u)=0$ correspond
to solutions of $\eqref{Eq_Gelfand-perturb}.$ For any neighborhood
$\mathcal{U} \subset X$ of $0_E$, the function $u=0$ is the unique
solution to \eqref{Eq-Eq-bif-Gelfand} with $\lambda=0$. Moreover,
$$ \deg(\mathcal{G}(0,\cdot),\mathcal{U},0_E)=\deg(I, \mathcal{U}, 0_E)=1,$$
and therefore from Theorem \ref{THM_global_unbounded} (see Appendix \ref{Sec_AppendixA}), there exists a
global curve of nontrivial solutions which emanates from $(0,0_E).$ Here
$\deg(\mathcal{G}(0,\cdot),\mathcal{U},0_E)$ stands for the classical
Leray-Schauder degree of $\mathcal{G}(0,\cdot)$ with respect to $\mathcal{U}$ and $0_E$; see e.g. \cite{Deimling,nir}.
From the maximum principle these solutions are positive in $\Omega$.
Since $u=0$ is the unique solution for $\lambda=0$ (up to a
multiplicative constant), the corresponding continuum of solutions is unbounded in
$(0,\infty)\times E$ according to Theorem \ref{THM_global_unbounded}.

From e.g.~\cite[Theorem 2.3]{Lions}, it is known that there exists a
minimal positive solution of \eqref{Eq_Gelfand-perturb} for all
$\lambda>0$; cf. ~also Proposition \ref{Prop_classical-results} below. Furthermore, there exists $\lambda^{\sharp}$
such that for {\mkrevb every $\lambda\geq \lambda^{\sharp}$, 
only one positive solution, $u_{\lambda}$,  of \eqref{Eq_Gelfand-perturb}, exists
 (cf. ~\cite{Castro_Shivaji}). The branch
$\lambda \mapsto u_{\lambda}$ is furthermore increasing on $[\lambda^{\sharp},\infty)$; see \cite{Amann} and
see Proposition \ref{Prop_classical-results} below.}

For $\lambda$ small enough, i.e.~when $0<\lambda\leq
\lambda_{\sharp}$ for some $\lambda_{\sharp}>0$, the same conclusions about the uniqueness of positive
solutions as well as about the monotony of the corresponding branch, are satisfied. The
problem is then to know what happens for $\lambda\in
(\lambda_{\sharp},\lambda^{\sharp})$.  {\mkrevb The aforementioned topological degree arguments 
may give some clues in that respect.} {\mkrevb For instance, since Theorem \ref{THM_global_unbounded} ensures that the solution set forms a continuum, then necessarily this continuum is $S$-like shaped\footnote{with possibly several  turning points not necessarily reduced to two.} in case of existence of three solutions for some $\lambda_{\s}\in(\lambda_{\sharp},\lambda^{\sharp})$.}

The determination of the exact shape of this continuum, for general domains,  is however
a challenging problem.
For instance it is known that for $\epsilon \geq 1/4$, the problem
\eqref{Eq_Gelfand-perturb} has in any dimension a {\it unique} {\mkrevb positive solution for every
$\lambda >0$ forming a monotone branch of
solutions as a function  of  $\lambda$; see e.g. ~\cite{BIS,Castro_Shivaji}.} However, if $d=2$ and $\Omega$ is the unit
open ball of $\mathbb{R}^2$, then  there exists $\epsilon^*>0$ such that for
$0<\epsilon< \epsilon^*$ the {\mkrevb continuum of solutions}  is exactly $S$-shaped  with exactly two turning points\cite{Du_Lou}. 
{\mkrevb This continuum may become nevertheless  more complicated than $S$-shaped when 
$\Omega$ is the unit  ball in higher dimension;  see \cite{Du} for $3\leq d\leq 9$.}

{\mkrevb In the one-dimensional case,} a lower bound of  the critical value $\epsilon^*>0$, for which the continuum of solutions  is exactly S-shaped, has been derived in \cite{Korman_Li}. 
It ensures in particular that $\epsilon^* > \epsilon_0$ with $\epsilon_0\approx 4.35$ when $\Omega=(-1,1)$ (\cite[Lem.~3.1]{Korman_Li}), which gives a rather sharp bound of
$\epsilon^*$ in that case,  since $\epsilon^*\leq \frac{1}{4}$ from the general results of \cite{BIS,Castro_Shivaji}. Numerical methods with guaranteed accuracy to enclose a double turning point  strongly suggest that {\mkrevb this theoretical lower bound} can be further improved \cite{Minamoto}. 

{\mkrevb Based on such numerical and theoretical results, it can be reasonably conjectured
that for $\Omega=(-1,1)$,} the
$\lambda$-bifurcation diagram does not present any turning point
(monotone branch) when $\epsilon >1/4$, whereas once $\epsilon <1/4$, an $S$-shaped
bifurcation takes place. {\mkrevb We observe thus on this example,} that a continuous change in the
parameter $\epsilon$ may lead to a qualitative change of its 
$\lambda$-bifurcation diagram on its whole: from a monotone curve to an $S$-shaped
curve as $\epsilon$ crosses 1/4 from above. 

It should be kept in mind however that the critical value of $\epsilon$ at which the $\lambda$-bifurcation diagram experiences a qualitative change, depends on the dimension and the shape of the domain. The numerical results of  \cite{Minamoto} indicate for instance that $\epsilon^* \in(0.238, 0.2396]$ when $\Omega$ is the unit {\mkrevb open} ball of $\mathbb{R}^3$.  In a similar fashion, the $\lambda$-bifurcation diagram does not become necessarily $S$-shaped as an $\epsilon$-critical value is crossed, depending on the shape of the domain and its dimension. The number  of positive solutions of
\eqref{Eq_Gelfand-perturb} may be indeed greater than three for some values
of $\lambda$ in dimension two, when $\Omega$ is the union of several touching balls; see \cite{Du,Dancer_domain_effect}. 
In other words, the critical perturbation that lead to a qualitative  change of the bifurcation diagram depends on the dimension; dimension-dependence that will appear also to play a key role under the more general setting of Theorem \ref{THM_Main}; see also Sect.~  \ref{concl_rem}.

%%%%%

\subsection{Classical structural stability for dissipative semilinear parabolic problems}\label{Sec_struc_stab}

{\mkrevb The qualitative change discussed above  of  the global $\lambda$-bifurcation diagram  is reminiscent, for $\Omega=(-1,1)$, with the so-called {\it cusp bifurcation}
observed in two-parameter families  of autonomous ordinary differential equations (ODEs) \cite{Kutz_book}.}

{\mkrevb Recall that the normal form of a cusp-bifurcation is given by $\dot{x} = \beta_1 + \beta_2 x
-x^3$, where $x \in {\mathbb R},$ and $\beta=(\beta_1,\beta_2) \in
{\mathbb R}^2.$ Two bifurcations curves, $\gamma_{+}$ and  $\gamma_{-}$, are naturally associated with this normal form. Each point of these curves, corresponds to a collision and disappearance of two equilibria, namely a {\it saddle-node
bifurcation}; see \cite{Kutz_book}.}

These two curves divide the parameter
plane into two regions:  inside the ``dead-end'' formed by $\gamma_{+}$ and
$\gamma_{-}$, there are three steady states, two stable and one unstable,
and outside this corner, there is a single steady state, which is
stable.  A crossing of the {\it cusp point}, $\beta=(0,0)$, from outside
the ``dead-end,''  leads to  an unfolding of singularities \cite{arnold1981singularity,Arnold_Geom,Church_Timourian,golubitsky1985singularities}
which consists more exactly to an unfolding of three steady states from a single stable
equilibrium; see also \cite{Kutz_book}.

The qualitative changes described at the end of the previous section
may be therefore interpreted in that terms; see also \cite[Fig. ~1]{minamoto2007numerical}. Singularity theory is a natural framework to study the effects on the 
bifurcation diagram of small perturbations or imperfections to a given {\it static} model \cite{golubitsky1979theory,golubitsky1985singularities}. 
In that spirit, geometric connections between a double turning point and a cusp point have been discussed for certain nonlinear elliptic problems  in e.g.~\cite{BCT88,brezzi1982numerical,moore1980,spence1982non}.  However,  a {\mkrevb general understanding} of the effects of arbitrary perturbations on bifurcation diagrams {\mkrevb remains a challenging problem}, especially when the perturbations are not necessarily smooth; see however \cite{Dancer_domain_effect2,Hen05} for related issues.

Complementarily, it is tempting to describe the aforementioned qualitative changes  in terms of structural instability such as
encountered  in classical dynamical systems theory \cite{abraham1987,Arnold_Geom,TOPOSmale}. Nevertheless, as will be explained in Section \ref{Sec_top_rob}, such topological 
ideas have to  be recast into a formalism which takes into account the functional setting in which the parabolic and corresponding  elliptic problems are considered; see Definitions \ref{Def_Admissible},
 \ref{DEF_Equiv_Rel-family} and \ref{Def_struct-stab} below. 
 
{\mkrevb This formalism  will turn out to be particularly suitable for problems such as arising in  combustion theory or chemical kinetics \cite{F} for which the associated  semigroups are not necessarily dissipative while still exhibiting finitely many local attractors which  attract the trajectories that remain bounded.  To better appreciate this distinction with the standard theory, we recall briefly below the notion of structural stability such as encountered for dissipative infinite-dimensional systems.}

{\mkrevb  Originally formulated for finite-dimensional dynamical systems \cite{Andornov_Ponrtyagin}, the notion of structural stability has been extended to infinite-dimensional dynamical systems, mainly {\it dissipative}.}
As a rule of thumb for such dynamical systems, one investigates structural stability of the semiflow restricted to
a compact invariant set, usually the {\mk global attractor}, rather than the flow in the
original state space \cite[Definition 1.0.1]{Hale_al'84}; an
exception can be found in e.g. ~\cite{Lu} where the author considered the
semiflow in a neighborhood of the global attractor.

In the context of
reaction-diffusion problems, the problem of structural stability is concerned with,
\begin{equation}\label{Eq_parab0}
\left\{
\begin{array}{l}
\partial_t u -\Delta u=g(u), \; \textrm{ in } \Omega,\;  \;g\in
C^1(\mathbb{R},\mathbb{R}),\\
\hspace{1.85em} u|_{\partial{\Omega}}=0,
\end{array}
\right.
\end{equation}
that is assumed to generate a semigroup $\{S(t)\}_{t\geq 0}$ for which 
a {\it global attractor} $\mathcal{A}_g$, in some Banach space $X$,  exists
 \cite{Brunovsky_Polacik, FiedlerRocha99, Hale_al'84,Lu}.

 Within this context, the structural stability problem may  be formulated as the existence problem of an
homeomorphism $H:\mathcal{A}_g \rightarrow \mathcal{A}_{\widehat{g}}$
for arbitrarily small perturbations $\widehat{g}$ of $g$ in some
topology $\mathcal{T}$ on {\mk $C^1(\mathbb{R},\mathbb{R})$}, that aims to
satisfy the following properties
\begin{subequations}\label{Eq_conjug}
\begin{align}
&\mathcal{A}_{\widehat{g}} \mbox{ is a global attractor in } X \mbox{ of } \{\widehat{S}(t)\}_{t\in \mathbb{R}^+}, \mbox{ and }\\
&\forall \; t\in \mathbb{R},\; \forall\; \phi \in \mathcal{A}_g,\; H
(S(t) \phi)=\widehat{S}(t)H(\phi),\label{conjugacy_relation}
\end{align}
\end{subequations}
where $\{\widehat{S}(t)\}_{t\in \geq 0}$ denotes the semigroup generated
by 
\bes
u_t-\Delta u=\widehat{g}(u), \; u|_{\partial \Omega}=0. 
\ees 

The topology $\mathcal{T}$ may be chosen to be for instance the compact-open topology or
the finer topology of Whitney.\footnote{See \cite{TOPOhirch} for general
definitions of these topologies, and see \cite{Brunovsky_Polacik} for
{\mk issues concerning the genericity of structurally stable reaction-diffusion
problems of type \eqref{Eq_parab0}, making use} of the Whitney topology.} Note that in
\eqref{conjugacy_relation}\footnote{ Note that \eqref{conjugacy_relation} may be substituted by the more general condition requiring that for all $t\in
\mathbb{R},$  and for all $\phi \in \mathcal{A}_g,\; H (S(t)
\phi)=\widehat{S}(\gamma(t,\phi))H(\phi)$, with $\gamma:\mathbb{R}\times
\mathcal{A}_g\rightarrow  \mathbb{R}$ an increasing and continuous
function of the first variable. Although this condition is often encountered in the
literature, its use is not particularly required when with the questions considered in the present article; see Remark \ref{Rem_Autre-Equivalence_relation} below.},
the restriction of the dynamics to the {\mk global attractor,} allows {\mkrevb for 
backward trajectories} onto the global attractor giving rise to genuine flows onto
the global attractor; see {\mk e.g. \cite{FiedlerRocha99,Robinson1}.}

 Once a parabolic equation generates 
a semigroup,  a necessary condition to exhibit a global attractor (in some Banach
$X$) is to satisfy a {\it dissipation property}, i.e.~to verify
the existence of an absorbing ball  in $X$ for this semigroup; see e.g.~\cite[Theorem 3.8]{Ma_Wang_Zhong}.

However, such a working assumption may be viewed as too restrictive. {\mkrevb  As mentioned above, in many applications although blow-up in finite or
infinite time may occur for certain trajectories, many other trajectories are typically attracted by local attractors depending on the ``energy'' of their initial data; see \cite{Bebernes_Eberly,Ben_Gal,Caz_Haraux_book, F,Fil05,quittner2007superlinear}.

Furthermore given a parameterized family of elliptic problems subject to perturbations, if one wants to translate a qualitative change of its bifurcation diagram  into dynamical terms for the corresponding parabolic problems, one has to take into account the possible discrepancies of regularity between the (weak) steady state solutions and the semigroup equilibria. The next section introduces a framework to deal with these issues. 
}

\subsection{Topological robustness for  general families of semilinear 
parabolic problems}\label{Sec_top_rob}
To deal with the problem of topological equivalence between families of semigroups which {\mkrevb may exhibit non-global dissipative properties}, we start by introducing several intermediate
concepts allowing for taking into account the possible discrepancies between the functional settings in which the parabolic and corresponding  elliptic problems are well-posed; see Definitions \ref{Def_Admissible},
 \ref{DEF_Equiv_Rel-family} and \ref{Def_struct-stab} below. 
Throughout this section we illustrate these concepts on some standard semilinear parabolic and elliptic problems.

 Let us first consider a parameterized family
$\mathfrak{F}_f:=\{f_{\lambda}\}_{\lambda \in \Lambda}$ of functions
$I\rightarrow \mathbb{R}$, where $\Lambda$ is a metric space, and
$I$ is an unbounded interval of $\mathbb{R}$.  We are concerned with 
 the associated parameterized family of semilinear parabolic problems,
\begin{equation}\label{Eq_family-parab}\tag{$\mathcal{P}_{f_{\lambda}}$}
\left\{
\begin{array}{l}
\partial_tu-\Delta u=f_{\lambda}(u), \quad \mbox{ in } \; \Omega,\\
\; \; \; \; \; \;  u=0, \quad \quad \qquad \mbox{ on } \; \partial
\Omega,
\end{array} \right.
\end{equation}
where $\Omega$ is an open bounded subset of $\mathbb{R}^d$, with
additional regularity  assumptions on its boundary and $f_{\lambda}$ when needed. 

In  general, these problems may generate a family of semigroups acting on a functional space $X$ that  does not necessarily agree 
with the functional space $Y$ on which the (weak) solutions of
\begin{equation}\label{Eq_family-ellip}
\left\{
\begin{array}{l}
-\Delta u=f_{\lambda}(u), \qquad \mbox{ in } \; \Omega,\\
\; \; \; \; \; \;  u=0, \quad \quad \qquad \mbox{ on } \; \partial
\Omega,
\end{array} \right.
\end{equation}
{\mkr exist. As shown in Example \ref{Rem_(ii)-(iii)} below,  such situations arise when weak solutions to  \eqref{Eq_family-ellip} do not necessarily correspond to equilibria of the semigroup} associated with \eqref{Eq_family-parab}.  
{\mkr These considerations} lead us naturally to  introduce the following definition that {\mkrevb makes precise the class of problems \eqref{Eq_family-parab} we consider hereafter.}

\begin{Def}\label{Def_Admissible}
Let $\Lambda$ be a metric space. {\mkr Let $Y$ be a Banach
space  and $\Omega$ be an open bounded subset of $\mathbb{R}^d$, such
that \eqref{Eq_family-ellip} makes
sense in $Y$.}  

{\mkr Given a Banach space $X$,} a family of functions, $\mathfrak{F}_f:=\{f_{\lambda}\}_{\lambda \in
\Lambda^*}$, is be said to be $(X;Y)$-{\mkr compatible} relatively to
$\Lambda^*\subset \Lambda$ and $\Omega$, if there exists a subset
$\Lambda^* \subset \Lambda$, such that for all $\lambda\in
\Lambda^*$ the following properties are satisfied:
\begin{itemize}
\item[(i)] There exists a nonempty subset
$D(f_{\lambda})\subset X$ such that \eqref{Eq_family-parab} generates a
semigroup $\{S_{\lambda}(t)\}_{t\geq 0}$ on $D(f_{\lambda})$.
\item[(ii)] The set $\mathcal{V}_{f_{\lambda}}:=\{u\in Y\;:\; -\Delta u=f_{\lambda}(u), \; u|_{\partial
\Omega}=0\}$  is non-empty.
\item[(iii)] The set $\mathcal{E}_{f_{\lambda}}$ of equilibria of $\{S_{\lambda}(t)\}_{t\geq 0}$, satisfies
$$\mathcal{E}_{f_{\lambda}}:=\{\phi \in D(f_{\lambda})  \;:\; S_{\lambda}(t)\phi
=\phi, \; \forall \; t\geq 0\}=\mathcal{V}_{f_{\lambda}}.$$
\end{itemize}
{\mkk If instead of (iii), 
\be\label{Eq_weak_comp}
\overline{\mathcal{E}_{f_{\lambda}}}^X=\mathcal{V}_{f_{\lambda}}, \mbox{ with } \mathcal{E}_{f_{\lambda}}\varsubsetneq\ \mathcal{V}_{f_{\lambda}}, 
\ee
then 
$\mathfrak{F}_f$ is be said to be weakly $(X;Y)$-compatible relatively to
$\Lambda^*\subset \Lambda$ and $\Omega$. }

\end{Def}

\begin{Rmk}
When the domain $\Omega$ is {\mkr clear from the context}, we simply say that a family of functions is  $(X;Y)$-compatible without referring to $\Omega$.
We will also often say  that the  family of elliptic problems \eqref{Eq_family-ellip} is $(X;Y)$-compatible, when 
the corresponding family of function $\{f_{\lambda}\}$ is $(X;Y)$-compatible.  
\end{Rmk}

We first provide an example of a family of {\it superlinear} elliptic problems that is not $(C^1(\overline{\Omega});H_0^1(\Omega))$-compatible, but {\mkr only} weakly 
 $(C^1(\overline{\Omega});H_0^1(\Omega))$-compatible. 

\begin{Ex}\label{Rem_(ii)-(iii)}
It may happen that $\mathcal{E}_{f_{\lambda}} \neq
\mathcal{V}_{f_\lambda}$ for some $\lambda\in\Lambda^*$. The Gelfand problem \cite{Gelfand_63,Fuj69},
\be\label{Eq_Gelfand}
-\Delta u=\lambda e^u, \; u|_{\partial B_1(0)}=0, 
\ee
where $B_1(0)$ is a unit
ball of $\mathbb{R}^d$ with $3\leq d\leq 9$, is an illustrative example of such a distinction that may arise between the set of equilibrium points and the set of steady states, depending on the functional setting adopted. 

{\mkr In that respect, let us first recall} that for $Y=H^1_0(B_1(0))$
there exists  $\lambda^*>0$ such that for $\lambda
>\lambda^*$ there is no solution to \eqref{Eq_Gelfand}, even in a very
weak sense \cite{Brezis_al'96}, whereas for
$\lambda\in[0,\lambda^*]$ there exists at least a solution (in $Y$) so that $\mathcal{V}_f \neq \emptyset$; see
\cite{Brezis_Vazquez} and Proposition \ref{Prop_classical-results}
below. 

{\mkr In what follows we denote by $A_p$ the (closed) Laplace operator considered as an unbounded
operator on $L^{p}(B_1(0))$ under Dirichlet conditons, with domain 
\bes
D(A_p)=W^{2,p}(B_1(0))\cap
W_0^{1,p}(B_1(0)); 
\ees
see \cite[Sect.~7.3]{PAZ}.} 

Let us {\mkr now} take $\Lambda^*$ to be $[0,\lambda^*]$ {\mkr and} let us choose
$X$ to {\mkr be the following subspace constituted by  radial functions}
\be
X:=\{\varphi(r) \;:\; \varphi \in D(A_p^{\beta})\},
\ee
where {\mkr $D(A_p^{\beta})$ denotes the domain of $A_p^{\beta}$, the fractional power of $A_p$, where $0<\beta \leq 1$; see e.g.~ \cite[Sect.~2.6]{PAZ} and \cite[Sect.~1.4]{Hen81}.}

For $p > d$ and $1 > \beta >
(d+p)/(2p)$, {\mkr it is known} that $D(A_p^{\beta})$ is compactly embedded in
$C^1(\overline{B_1(0)})$ {\mkr \cite[Thm.~1.6.1]{Hen81}, and thus {\mkr $X\hookrightarrow C^1(\overline{B_1(0)})$}}.
Then for {\mkr  any $\lambda \in [0, \lambda^*]$ and for such a choice of $p$ and
$\beta$}, the parabolic problem 
$\eqref{Eq_family-parab}$ is well posed in $X$  with $f_{\lambda}(x)=\lambda \exp(x)$, see \cite{Caz_Haraux_book,SellYou,Lunardi_book}.

{\mkr As a consequence, by introducing}
\be\label{Eq_equilibriaset}
D(f_{\lambda}):=\{u_0 \in X\;:\;
u_{\lambda}(t;u_{0}) \mbox{ exists
 for all }t>0, \mbox{ and }\underset{t>0}\sup\;\|A_p^{\beta}
u_{\lambda}(t;u_{0})\|_p<\infty\},
\ee 
 a {\mkr nonlinear} semigroup $\{S_{\lambda}(t)\}_{t\geq 0}$ on
$D(f_{\lambda})$ {\mkr can be defined as follows}
\be
S_{\lambda}(t)u_0 :=u_{\lambda}(t;u_0), \; t \geq 0, \; u_0 \in D(f_{\lambda}),
\ee
where $u_{\lambda}(t;u_0)$ denotes the solution of
\eqref{Eq_family-parab} emanating from $u_0$ at $t=0$.

However {\mkr the property} (iii) of Definition
\ref{Def_Admissible} is not {\mkr verified} here. 
{\mkr Indeed, for}
$\lambda=\lambda^{\sharp}=2(d-2) \in (0,\lambda^*)$ there exists {\mkr in $H_0^1(B_1(0))$} an
unbounded solution of the Gelfand problem \eqref{Eq_Gelfand}  \textemdash\, in the weak sense of \cite{Brezis_al'96}  \textemdash\, given by
\bes
u_{\lambda^{\sharp}}(x):=-2\log \|x\|,
\ees
see \cite{Brezis_Vazquez}.

{\mkr This solution does not belong to} $D(f_{\lambda})\subset X \subset
C^1(\overline{B_1(0)})$ and in particular to
$\mathcal{E}_{f_{\lambda^{\sharp}}}$, {\mkr the set  of equilibria of $S_{\lambda}(t)$ in $D(f_{\lambda})$ given by \eqref{Eq_equilibriaset}}. 

Therefore the family
\be
\mathfrak{F}_{\mbox{exp}}=\{x\mapsto \lambda e^x, x\geq 0, \; \lambda \in [0,\lambda^*]\},
\ee
is not $(C^1(\overline{B_1(0)});H_0^1(B_1(0)))$-compatible relatively
to $[0,\lambda^*]$ where $B_1(0)$ is the unit open ball of
$\mathbb{R}^d$, for $3\leq d\leq 9$. 

Nevertheless this family is weakly $(C^1(\overline{B_1(0)});H_0^1(B_1(0)))$-compatible relatively to $[0,\lambda^*]$, {\mkr in the sense of Definition \ref{Def_Admissible}. This property} results from the fact that the singular steady state $u_{\lambda^{\sharp}}$ can be approximated by a sequence of equilibria in $X$ for the relevant topology \cite{Brezis_Vazquez,Joseph_Lund}, so that in particular condition \eqref{Eq_weak_comp} is verified.  

\end{Ex}

The following proposition identifies a broad class of families of {\it sublinear} elliptic {\mkr problems} which are $(C_0^{0,2\alpha}([0,1]);C^{2}([0,1]))$-compatible for $\alpha \in
(\frac{1}{2},1)$.  

\begin{proposition}\label{Prop_identif_admissible}
Let us consider a function $f:[0,\infty)\rightarrow (0,\infty)$ that
satisfies the following conditions:
\begin{itemize}
\item[(G$_1$)] $f$ is locally Lipschitz, and such that for all $\sigma>0$, the following properties hold: 
\begin{itemize}
\item[(i)] $f\in C^{\theta}([0,\sigma])$, for some $\theta \in (0,1)$ (independent of $\sigma$),   and
\item[(ii)] $\exists \;  \omega(\sigma)>0$ such that 
\bes
f(y)-f(x)>-\omega(\sigma)  (y-x), \; 0\leq x<y\leq \sigma.
\ees
\end{itemize}

\item[(G$_2$)] $x \mapsto f(x)/x$ is strictly decreasing on
$(0,\infty)$.
\item[(G$_3$)] $\underset{x\rightarrow \infty}\lim (f(x)/x)=b,$ with
$b\geq 0$.
\end{itemize}
Let us define
$a=\underset{x\rightarrow 0}\lim (f(x)/x)$, and $\Lambda^*:=(\frac{\lambda_1}{a},
\frac{\lambda_1}{b}).$

If $a<\infty$, then $\mathfrak{F}_f=\{\lambda f\}_{\lambda \in \Lambda^*}$ is
$(C_0^{0,2\alpha}([0,1]);C^{2}([0,1]))$-compatible  relatively to
$\Lambda^*$, for $\alpha\in(\frac{1}{2},1)$.
\end{proposition}

\begin{proof}
This proposition is a direct consequence of the theory of analytic semigroups
\cite{Lunardi_book,PAZ,SellYou,Taira_book} and the theory of sublinear elliptic equations \cite{Brezis_Oswald}.

Consider  $\Lambda=[0,\infty)$, and $f_{\lambda}=\lambda f$, for
$\lambda \in [0,\infty)$. Then from \cite[Theorem
5]{Taira} which generalizes the ``classical'' result of 
\cite[Theorem 1]{Brezis_Oswald}, we have that
$$- \partial_{xx}^2 u=\lambda f(u),\; u(0)=u(1)=0,$$ has a unique solution
$u\in C^{2}([0,1)])$ if and only if 
%{\attn $ \lambda_1/a<\lambda
%<\lambda_1/b $ }
{\mkk 
\be  
 \frac{\lambda_1}{a}<\lambda
< \frac{\lambda_1}{b},
\ee
 }
where $\lambda_1$ is the first eigenvalue of $-\partial_{xx}^2$ with Dirichlet condition.

Let us consider $\Lambda^*:=(\frac{\lambda_1}{a},
\frac{\lambda_1}{b}).$ The  realization of the Laplace operator
$A=-\partial_{xx}^2$ in $X=C([0,1])$ with domain,
\be\label{D(A)}
D(A)=C^{0,2\alpha}_0([0,1]):=\{u\in C^{0,2\alpha}([0,1])\; :
\;u(0)=u(1)=0 \},
\ee
 is sectorial for $\alpha \in (\frac{1}{2},1)$,
and therefore generates an analytic semigroup on $X$; see
\cite{Lunardi_book}.

The theory of analytic semigroups shows that under the aforementioned 
assumptions on $f$, for every $u_0\in C_0^{0,2\alpha}([0,1])$, there
exists a unique solution {\mkrevb $u_{\lambda} \in
C^{1}((0,\tau_{\lambda}(u_0));C^{2}([0,1]))$} of \eqref{Eq_family-parab}
defined on a maximal interval $[0,\tau_{\lambda}(u_0)),$ with
$\tau_{\lambda}(u_0)>0$  (and $f_{\lambda}=\lambda f$); see e.g. \cite[Proposition
6.3.8]{Lunardi_Lecture-notes}. Since our assumptions on $f$ imply that
there exists $C>0$ such that $0\leq f(x)\leq C (1+x)$ for all $x\geq 0$,
from e.g.~\cite[Proposition 6.3.5]{Lunardi_Lecture-notes} we can deduce
that $\tau_{\lambda}(u_0)=\infty.$

Let us introduce now,
\begin{equation}\label{Domain_Semi-group}
D(f_{\lambda}):=\{u_0\in   C_0^{0,2\alpha}([0,1])\;:\;
\underset{t>0}\sup\;\|u_{\lambda}(t;u_{0})\|_{C^{2}([0,1])}<\infty\},
\end{equation}
 then $S_{\lambda}(t): D(f_{\lambda}) \rightarrow D(f_{\lambda})$, defined by
$S_{\lambda}(t)u_0=u_{\lambda}(t;u_0)$ is well defined for all
$t\geq 0$, and for all $u_0 \in D(f_{\lambda})$. {\mkrevb From the existence and uniqueness properties of the solutions, we deduce that 
$\{S_{\lambda}(t)\}_{t \geq 0}$ is  a (nonlinear) semigroup on
$D(f_{\lambda}),$ in the sense that $S_{\lambda}(t) \in C(D(f_{\lambda}),D(f_{\lambda})),$
\be
S_{\lambda}(t+s)=S_{\lambda}(t)\circ S_{\lambda}(s), \; \; \forall\, t,s \geq 0,
%S_{\lambda}(0)&=\mbox{Id}_{C^{2}([0,1])},
\ee 
and that each trajectory $t\mapsto S_{\lambda}(t)u_0$ is continuous in $D(f_{\lambda})$.}

It is now easy to verify from what precedes that (ii) and (iii) of
Definition \ref{Def_Admissible} are satisfied. We have thus proved
that $\mathfrak{F}_f=\{\lambda f\}_{\lambda \in \Lambda^*}$ is
$(C_0^{0,2\alpha}([0,1]);C^{2}([0,1]))$-compatible  relatively to
$\Lambda^*$, for $\alpha\in(\frac{1}{2},1)$.
\end{proof}

\begin{Rmk}
Let us remark that if  we assume furthermore that $\lambda b> \lambda_1^{-1}$, it can be then proved\footnote{Based on Lyapunov functions techniques
\cite{Caz_Haraux_book} and the non-increase of lap-number of
solutions for scalar semilinear parabolic problems
\cite{Matano_lap-number}.} that there exists at least one solution $u$ to
$\eqref{Eq_family-parab}$ emanating from some $u_0\in
C_0^{0,2\alpha}([0,1])$ for which $u$ does not remain in any bounded set for
all time \cite[Lemma 10.1, Remark 10.2]{Ben_Gal}. {\mkrevb Such a trajectory becomes
unbounded in infinite time.} It is the possible occurrence of such a phenomenon that motivated  to include a boundedness requirement in the definition of  $D(f_{\lambda})$ in \eqref{Domain_Semi-group}. 
\end{Rmk}

\begin{Ex}\label{Rem_Arrenus}
 Let $g_{\epsilon}(x)=\exp(x/(1+\epsilon x))$.
A simple calculation shows that for $x\neq 0$,
\bes
\Big(\frac{g_{\epsilon}(x)}{x}\Big)'=-\frac{\exp(\frac{x}{1+\epsilon x})}{x^2(1+\epsilon x)^2} (\epsilon^2 x^2 +(2 \epsilon-1) x +1),
\ees
which implies in particular that $g_{\epsilon}(x)/x$ is strictly decreasing for all $x > 0$ if
$\epsilon> 1/4$. Note also that Condition (G$_1$) of Proposition \ref{Prop_identif_admissible}
 is satisfied,  and that  $b=0$ and $a = \infty$ in this case.

A semigroup can still be defined (for each $\lambda \in (0,\infty)$) on the subset $D(\lambda g_{\epsilon})$ such as given in \eqref{Domain_Semi-group} with $f_{\lambda}=\lambda g_{\epsilon}.$ From the proof of 
 Proposition \ref{Prop_identif_admissible},  it is then easy to deduce that the family $\{\lambda
g_{\epsilon}\}_{\lambda \in (0,\infty)}$ is in fact 
$(C_0^{0,2\alpha}([0,1]);C^{2}([0,1]))$-compatible relatively to
$(0,\infty)$,  for $\alpha \in (\frac{1}{2},1)$  and $\epsilon>1/4.$

\end{Ex}

Hereafter, $X$  and $Y$ are two Banach spaces with respective
norms denoted by $\|\cdot\|_X$ and $\|\cdot\|_Y$; and $\Omega$ denotes
an open bounded subset of $\mathbb{R}^d$, such that  the following elliptic problem
\bea
-\Delta u&=f_{\lambda}(u), \; \mbox{ in } \Omega,\\ 
u&=0, \qquad \mbox{ on }\partial \Omega,
\eea
makes sense in $Y$. We
introduce below a concept of topological equivalence between
families of semilinear parabolic problems for $(X;Y)$-compatible
families of nonlinearities.

\begin{Def}\label{DEF_Equiv_Rel-family}
Let $\Lambda$ be a metric space and $I$ be an unbounded interval of
$\mathbb{R}$. Let $\mathcal{N}(I,\mathbb{R})$ be a set of functions from $I$ to $\mathbb{R}$.
Consider two  families $\{f_{\lambda}\}_{\lambda \in \Lambda^*}$ and
$\{\widehat{f}_{\lambda}\}_{\lambda \in \widehat{\Lambda}^*}$ of
$\mathcal{N}(I,\mathbb{R})$, which are {\mkr both} $(X;Y)$-compatible relatively
to $\Lambda^*$ and $\widehat{\Lambda}^*$ respectively.  

For each
$\lambda\in \Lambda^*$ and $\lambda\in \widehat{\Lambda}^*$, one denotes by
 $\{S_{\lambda}(t)\}_{t\geq 0}$ and
$\{\widehat{S}_{\lambda}(t)\}_{t\geq 0}$, the semigroups acting on
$D(f_{\lambda})$ and  $D(\widehat{f}_{\lambda})$,  and associated with \eqref{Eq_family-parab} and  {\normalfont(}$\textcolor{blue}{\mathcal{P}_{\widehat{f}_{\lambda}}}${\normalfont )}, respectively.  {\mkr One denotes
finally by} $\mathfrak{S}_{f}$ and by $\mathfrak{S}_{\widehat{f}}$, the
respective family of such semigroups.

Then $\mathfrak{S}_f$ and $\mathfrak{S}_{\widehat{f}}$ are called
topologically equivalent if there exists an homeomorphism

$$H: \Lambda \times \underset{\lambda\in\Lambda^*}\bigcup D(f_{\lambda})
\rightarrow \Lambda \times \underset{\lambda \in
\widehat{\Lambda}^*}\bigcup D(\widehat{f}_{\lambda}),$$ such that
$H(\lambda,u)=(p(\lambda),H_{\lambda}(u))$ where $p$ and
$H_{\lambda}$ satisfy the following two conditions:
\begin{itemize}
\item[(i)] $p$ is an homeomorphism from $\Lambda^*$ to $\widehat{\Lambda}^*$,
\item[(ii)] for all $\lambda \in \Lambda^*$, $H_{\lambda}$ is an
homeomorphism from $D(f_{\lambda})$ to $D(\widehat{f}_{p(\lambda)})$,
such that,
\begin{equation}\label{Eq_Equiv_Rel-family}
\forall \; \lambda \in \Lambda^*, \; \forall \; u_0 \; \in
D(f_{\lambda}), \; \forall \; t>0, \; H_{\lambda} (S_{\lambda}(t)
u_0)=\widehat{S}_{p(\lambda)}(t)H_{\lambda}(u_0).
\end{equation}
\end{itemize}
In case of such an equivalence, the families of problems $\{\eqref{Eq_family-parab}\}_{\lambda\in
\Lambda^*}$ and $\{(\textcolor{blue}{\mathcal{P}_{\widehat{f}_{\lambda}}})\}_{\lambda\in
\widehat{\Lambda}^*}$ is also referred to as topologically equivalent.

\end{Def}

\begin{Rmk}\label{Rem_Autre-Equivalence_relation}
Note that the relation of  topological equivalence  given by
\eqref{Eq_Equiv_Rel-family} may be relaxed as follows,
\begin{equation}\label{Eq_Equiv_Rel-family2}
\forall \; \lambda \in \Lambda, \; \forall \; u_0 \; \in
D(f_{\lambda}), \; H_{\lambda} (S_{\lambda}(t)
u_0)=\widehat{S}_{p(\lambda)}(\gamma(t,u_0))H_{\lambda}(u_0),
\end{equation}
where $\gamma: [0,\infty)\times D(f_{\lambda}) \rightarrow
[0,\infty)$ is an increasing and continuous function of the first
variable. 

The equivalence relation \eqref{Eq_Equiv_Rel-family2} is known as the
topological {\it orbital} equivalence\footnote{Such as classically encountered in
finite-dimensional dynamical systems theory \cite{TOPOkat}}. It allows,  in particular,  for systems presenting periodic orbits of different periods, to be equivalent.\footnote{Avoiding in this way the so-called problem of modulii; see
\cite{Arnold_Geom,TOPOkat}.} 

{\mkrevb In contrast, the topological equivalence relation
\eqref{Eq_Equiv_Rel-family} excludes this possibility, which might be viewed as too restrictive for general semigroups, at a first glance. However, for
semigroups generated by semilinear parabolic equations over 
bounded domain, due to their gradient structure
\cite[Sect.~9.4]{Caz_Haraux_book}, this problem of modulii does
not occur since the $\omega$-limit set of each semigroup is
typically included into the set of its equilibria
\cite[Thm.~9.2.7]{Caz_Haraux_book}.} \end{Rmk}

\begin{Def}\label{Def_Fold-point}
Let $\mathfrak{S}_f$ be a family of semigroups as defined in
Definition \ref{DEF_Equiv_Rel-family}. Let $\mathcal{E}_f$  be the
corresponding family of equilibria, in the sense that,
\begin{equation}\label{Eq_family-Equilibria}
\mathcal{E}_f:=\{(\lambda,\phi_{\lambda}) \in \Lambda \times
D(f_{\lambda})\; : \; S_{\lambda}(t)\phi_{\lambda}=\phi_{\lambda},
\; \forall\; t\in (0,\infty)\}.
\end{equation}
Assume that $\Lambda$ is an unbounded interval of $\mathbb{R}$. A
{\it fold-point} on $\mathcal{E}_f$ is a point $(\lambda^*, u^*)\in
\mathcal{E}_f$, such that there exists a local continuous map
\bes
\mu:s\in (-\epsilon,\epsilon) \mapsto (\lambda(s),u(s)) \mbox{ for some }
\epsilon>0, 
\ees
verifying the following properties:
\begin{itemize}
\item[(F$_1$)] For all $s\in  (-\epsilon,\epsilon)$, one has $(\lambda(s),u(s)) \in
\mathcal{E}_f$, with $(\lambda(0),u(0))=(\lambda^*, u^*)$.
\item[(F$_2$)] The map $s\mapsto \lambda(s)$ has a unique extremum
 on $(-\epsilon, \epsilon)$ attained at $s=0$.
\item[(F$_3$)] There exists $r^*>0$
such that for all $0<r<r^*,$ the set 
\bes
\partial\mathfrak{B}((\lambda^*, u^*);r)
\bigcap \{\mu(s), \; s\in (-\epsilon,\epsilon) \},
\ees
 has cardinal two;
where 
\be\label{Eq_B}
\mathfrak{B}((\lambda^*, u^*);r):=\{(\lambda,u)\in
\mathbb{R}\times D(f_{\lambda}), \; : \;
|\lambda-\lambda^*|+{\mkr \|u-u^*\|_X}<r\}.
\ee
\end{itemize}
\end{Def}

\begin{Def}\label{Def_Multiple-point}
Let $\mathfrak{S}_f$ be a family of semigroups as defined in
Definition \ref{DEF_Equiv_Rel-family}. Let $\mathcal{E}_f$  be the
corresponding family of equilibria given by
\eqref{Eq_family-Equilibria}. Assume that $\Lambda$ is an
unbounded interval of $\mathbb{R}$. Let $n$ be an integer such that
$n\geq 3$. A {\it multiple-point} with $n$ branches on
$\mathcal{E}_f$ is a point $(\lambda^*, u^*)\in \mathcal{E}_f$, such
that there exists at most $n$ local continuous map 
\bes
\mu_i:s\in
(-\epsilon_i,\epsilon_i) \mapsto (\lambda_i(s),u_i(s)) \mbox{ for some }
\epsilon_i>0, \; i\in\{1,...,n\},
\ees
 verifying the following
properties:
\begin{itemize}
\item[(G$_1$)] $\mu_i\neq \mu_j$ for all $i\neq j$.
\item[(G$_2$)] For all $i\in\{1,...,n\}$, and for all $s\in  (-\epsilon_{i},\epsilon_i)$, one has $(\lambda_i(s),u_i(s)) \in
\mathcal{E}_f$, with $(\lambda_i(0),u_i(0))=(\lambda^*, u^*)$.
%\item[(G$_2$)] $s\mapsto \lambda(s)$ has a unique maximum or minimum at
%$s=0$ on $(-\epsilon, \epsilon)$.
\item[(G$_3$)] There exists $r^*>0$
such that for all $0<r<r^*,$  the set 
$$\partial\mathfrak{B}((\lambda^*, u^*);r)
\bigcap \underset{i\in \{1,...,n\}}\bigcup \{\mu_i(s), \; s\in
(-\epsilon_i,\epsilon_i) \},$$ 
has cardinal  $n$, where
$\mathfrak{B}((\lambda^*, u^*);r)$ is as given in \eqref{Eq_B}.
%introduced in Definition
%\ref{Def_Fold-point}.
\end{itemize}
\end{Def}

\begin{Rmk}
The terminologies of Definitions \ref{Def_Fold-point} and 
\ref{Def_Multiple-point} regarding the singular points of $\mathcal{E}_f$ will be also adopted, when they apply, for the singular points of the solution set associated with the family of elliptic problems  \eqref{Eq_family-ellip}. 
\end{Rmk}

Based on these definitions, simple criteria of non-topological
equivalence between two families of semigroups can be then formulated.   The proposition below whose proof is left to the reader's discretion,  summarizes these criteria.

\begin{proposition}\label{Prop_criteria-noneq}
Assume $\Lambda$ is an unbounded interval of $\mathbb{R}$. Let
$\mathfrak{S}_f$ and $\mathfrak{S}_{\widehat{f}}$ be two families of
semigroups as defined in Definition \ref{DEF_Equiv_Rel-family}. Let
$\mathcal{E}_f$ and $\mathcal{E}_{\widehat{f}}$ be the corresponding
families of equilibria. Then $\mathfrak{S}_f$ and
$\mathfrak{S}_{\widehat{f}}$ are not topologically equivalent if one of
the following conditions are fulfilled.
\begin{itemize}
\item[(i)] $\mathcal{E}_f$ is constituted by a single unbounded continuum in $\Lambda \times X$, and $\mathcal{E}_{\widehat{f}}$ is the union of at least two disjoint
unbounded continua in $\Lambda \times X$.
\item[(ii)] $\mathcal{E}_f$  and
$\mathcal{E}_{\widehat{f}}$ are each constituted by a single continuum,
and the set of fold-points of $\mathcal{E}_f$ and
$\mathcal{E}_{\widehat{f}}$ are not in one-to-one correspondence.
\item[(iii)] $\mathcal{E}_f$  and
$\mathcal{E}_{\widehat{f}}$ are each constituted by a single continuum,
and there exists an integer $n\geq 3$ such that the set of
multiple-points with $n$ branches of $\mathcal{E}_f$ and
$\mathcal{E}_{\widehat{f}}$ are not in one-to-one correspondence.
\end{itemize}
\end{proposition}

{\mkrevb We are now in position to formulate our notion of {\it topological robustness} to small perturbations for family of semigroups which may exhibit a non-global dissipative behavior.  In that respect, a first requirement that is needed in practice concerns the stability of the $(X;Y)$-compatibility of the underlying family of nonlinearities, in order to stay, loosely speaking, within the same functional setting when a perturbation is applied. This is formulated in the following definition. }

\begin{Def}\label{Def_struct-stab}
Let $\Lambda$ be a metric space and $I$ be an unbounded interval of
$\mathbb{R}$. Let $\mathcal{N}(I,\mathbb{R})$ be a set of functions from the interval $I$ to $\mathbb{R}$ {\mkr endowed with  a topology $\mathcal{T}$.}
Consider a family $\mathfrak{F}_f=\{f_{\lambda}\}_{\lambda \in
\Lambda^*}$ of $\mathcal{N}(I,\mathbb{R})$ which is
$(X;Y)$-compatible relatively to $\Lambda^* \subset \Lambda$.

{\mkrevb Let $\mathcal{P}$ be an open subset of $\mathcal{N}(I,\mathbb{R})$ for the $\mathcal{T}$-topology.  The family  $\mathfrak{F}_f$  is said to be $\mathcal{T}$-stable with respect to perturbations in $\mathcal{P}$, if  there exist an interval $\Lambda' \supseteq \Lambda^\ast$ 
and a neighborhood $\mathcal{U}'_{\lambda}$ of $f_{\lambda}$  in the $\mathcal{T}$-topology such that
for any neighborhood $\mathcal{U}_{\lambda} \subset
\mathcal{U}'_{\lambda}$, we have 

$$\Big( \;\widehat{f}_{\lambda} \in
\mathcal{U}_{\lambda} \mbox{ and } \widehat{f}_{\lambda} -f_{\lambda} \in \mathcal{P}, \, \lambda \in  \Lambda'\Big)\Rightarrow\Big(
\{\widehat{f}_{\lambda}\}_{\lambda \in \Lambda'} \mbox{ is 
}(X;Y)\mbox{-compatible relatively to }\Lambda'
\Big).$$
 }
%\end{itemize}}
\end{Def}

\begin{Ex}
{\mkrevb Let us consider  $\mathcal{N}((0,\infty),\mathbb{R})$ endowed with the $C^0$-topology $\mathcal{T}$ of uniform convergence over compact sets.  Let us consider $f_{\lambda}=\lambda g_{\epsilon}$, with $g_{\epsilon}(x)=\exp(x/(1+\epsilon x))$,  and   $\lambda\in \Lambda=\Lambda^\ast=(0,\infty)$. 

We saw in Example \ref{Rem_Arrenus} that the corresponding family, $\mathfrak{F}=\{f_{\lambda}\}_{\lambda \in\Lambda}$, is 
$(C_0^{0,2\alpha}([0,1]);C^{2}([0,1]))$-compatible relatively to
$\Lambda$ for $\alpha \in (\frac{1}{2},1)$ and $\epsilon >1/4$. 

Let $\mathcal{P}$ be the set of functions $\varphi$ with compact support such that $\widehat{g}:=g_{\epsilon}+\varphi$ is locally Liptchitz and  satisfies the rest of assumptions of  Proposition \ref{Prop_identif_admissible}. 
This set is non empty. Indeed, if we consider $0<m<M$, $r=\displaystyle  \lambda \frac{g_{\epsilon}(M)-g_{\epsilon}(m)}{M-m}$, and $\varphi$ given by
\bea
\varphi(x)&=r(x-m)+\lambda(g_{\epsilon}(m)-g_{\epsilon}(x)), \mbox{ for } x\in (m,M),\\
\varphi(x)&=0,  \mbox{ otherwise},
\eea
then the function $g_{\epsilon}+\varphi$ satisfies the desired assumptions. Furthermore this perturbation can be made as close as 
desired to  $g_{\epsilon}$ (in the aforementioned $C^0$-topology $\mathcal{T}$) by reducing the size of the interval  $(m,M)$, accordingly.

Now since the assumptions of Proposition \ref{Prop_identif_admissible} are satisfied for any $g_{\epsilon}+\varphi$ with $\varphi \in \mathcal{P}$, we conclude that $\mathfrak{F}'=\{\lambda (
g_{\epsilon} +\varphi)\}_{\lambda \in\Lambda}$ is 
$(C_0^{0,2\alpha}([0,1]);C^{2}([0,1]))$-compatible relatively to
$\Lambda'=(0,\infty)$ for $\alpha \in (\frac{1}{2},1)$. 
In other words,  $\mathfrak{F}$ is $\mathcal{T}$-stable with respect to perturbations in $\mathcal{P}$, for $\epsilon >1/4$. 

Note that in the proof of Corollary \ref{Main_cor} below, the family $\mathfrak{F}$ is shown  to be $\mathcal{T}$-stable for another class of perturbations than considered here, emphasizing thus that a given family can be $\mathcal{T}$-stable with respect to different type of perturbations. 
}
\end{Ex}

The desired notion of {\it topological robustness} to small perturbations and {\mkr the} related  notion of {\it topological instability} can be then formulated as follows. 

\begin{Def}\label{topo_robustness}
{\mkrevb Let us consider the setting of Definition  \ref{Def_struct-stab}. For each $\lambda$, one denotes by $\{S_{\lambda}(t)\}_{t\geq 0}$ (resp.~$\{\widehat{S}_{\lambda}(t)\}_{t\geq 0}$)
the semigroup acting on $D(f_{\lambda})$ (resp.~$D(\widehat{f}_{\lambda})$), given a function $f_{\lambda}$ (resp.~$\widehat{f}_{\lambda}$). One denotes also by $\mathfrak{S}_{f}$  and $\widehat{\mathfrak{S}}_{f}$
the corresponding family of semigroups generated respectively by \eqref{Eq_family-parab} and {\normalfont (}$\textcolor{blue}{\mathcal{P}_{\widehat{f}_{\lambda}}}${\normalfont )}. 

}

{\mkrevb In case where  $\mathfrak{F}_f$ is  $\mathcal{T}$-stable, we say furthermore  that $\mathfrak{S}_{f}$ is {\mkk $\mathcal{T}$-topologically robust} in $X$
with respect to perturbations in $\mathcal{P}$ for  the $\mathcal{T}$-topology, if there exists a
neighborhood $\mathcal{U}'_{\lambda}$ of $f_{\lambda}$ such that 
for any neighborhood $\mathcal{U}_{\lambda} \subset
\mathcal{U}'_{\lambda}$, we have over some interval $\Lambda' \supseteq \Lambda^\ast$,
\be\label{Eq_simple_equiv}
\Big(\widehat{f}_{\lambda} \in \mathcal{U}_{\lambda} \mbox{ and } \widehat{f}_{\lambda} -f_{\lambda} \in \mathcal{P} \Big)  \Rightarrow \Big(\mathfrak{S}_{\widehat{f}} \sim \mathfrak{S}_{f}\Big),
\ee
where  $\mathfrak{S}_{\widehat{f}} \sim \mathfrak{S}_{f}$ means  that $\mathfrak{S}_{\widehat{f}}$ and $\mathfrak{S}_{\widehat{f}}$ are topologically
equivalent in the sense of Definition \ref{DEF_Equiv_Rel-family}.}

 Given a  $\mathcal{T}$-stable  family $\mathfrak{F}_f$,   in case of violation of \eqref{Eq_simple_equiv},  then $\mathfrak{S}_{f}$ is {\mkrevb said to be topologically unstable  with respect to small perturbations in  $\mathcal{P} $ for the $\mathcal{T}$-topology.}
 \end{Def}

%%%%%%%%%%%%%%

\section{Topologically unstable families of semilinear parabolic problems: Main result}\label{Sec_Main}

{\mkrevb We are now in position to formulate the main result of this article, Theorem \ref{THM_Main}, regarding the topological instability of a broad class of semilinear parabolic problems. As the proof will show, the abstract framework introduced in the previous section allows us to relate these instabilities to local deformations\textemdash of the $\lambda$-bifurcation diagram of the corresponding elliptic problems\textemdash which occur when appropriate small perturbations are applied to the nonlinear term.} 

Figure \ref{fig:schematic} below depicts some typical bifurcation diagrams for which  Theorem \ref{THM_Main} predicts the apparition of either a multiple-point or a new fold-point on it when the nonlinearity is appropriately perturbed. It is worth mentioning that the parabolic problems corresponding to such bifurcation diagrams allow for a possible mixed dynamical behavior composed by finitely many local attractors and unbounded trajectories, justifying the revision of the standard notion of structural stability such as proposed in Section \ref{Sec_top_rob}.

To prepare the proof  of Theorem \ref{THM_Main}, one first recall some standard results regarding the solution set
of,
\begin{equation}\label{Eq_elliptic-pb}
\left\{
\begin{array}{l}
-\Delta u =\lambda g(u),\;  \textrm{ in } \Omega,\; \lambda\geq 0,\\
\hspace{1em}u|_{\partial \Omega}=0,
\end{array}
\right.
\end{equation}
summarized into the Proposition \ref{Prop_classical-results} below. The proof of this proposition, based on the use of sub- and
super-solutions methods, can be found in \cite[Theorem 3.4.1]{Nao_book}.

\begin{proposition}\label{Prop_classical-results}
Consider a {\mk locally Lipschitz function} $g:[0,\infty)\rightarrow (0,\infty).$ Let
$\Omega$ be a bounded, connected and open subset of $\mathbb{R}^d$.
Then there exists $0<\lambda^*\leq \infty$ with the following
properties.
\begin{itemize}
\item[(i)] For every $\lambda \in [0,\lambda^*)$, there exists a
unique minimal solution $\underline{u}_{\lambda}\geq 0,$
$\underline{u}_{\lambda} \in H_0^1(\Omega) \cap L^{\infty}(\Omega)$
of \eqref{Eq_elliptic-pb}. The solution $\underline{u}_{\lambda}$ is
minimal in the sense that any supersolution $v\geq 0$ of
\eqref{Eq_elliptic-pb} satisfies $v\geq \underline{u}_{\lambda}$.
\item[(ii)]The map $\lambda \mapsto \underline{u}_{\lambda}$ is
increasing  from $(0,\infty)$ to $H_0^1(\Omega) \cap
L^{\infty}(\Omega)$.
\item[(iii)]If $\lambda^* < \infty$ and $\lambda > \lambda^*$, then
there is no solution of \eqref{Eq_elliptic-pb} in $H_0^1(\Omega)
\cap L^{\infty}(\Omega)$.
\end{itemize}
If $\Omega$ is furthermore connected, then $\lambda^*=\infty$ if
$\frac{g(u)}{u}\underset{u\rightarrow \infty} \longrightarrow 0$,
and $\lambda^*<\infty$ if $\underset{u\rightarrow \infty}\lim\inf
\frac{g(u)}{u}>0.$
\end{proposition}

\begin{Rmk}\label{Rem_regul_growth}
 \cite[Theorem 3.4.1]{Nao_book} is in fact proved for functions $g$ which are $C^1$ but it is not difficult to adapt the arguments to the case of locally Lipschitz functions.
\end{Rmk}

We are now in position to prove our main theorem.

%%%%%%%%%
% % \vspace{-5ex}
%\begin{figure}[!hbtp]
%   \centering
%   \includegraphics[width=.8\textwidth,height=.35\textwidth]{u_norms.eps}
%  \vspace{-5ex}
%  \caption{{\footnotesize {\bf Schematic of typical situations dealt with Theorem \ref{THM_Main}}. The left panel corresponds to case (i), the right panel corresponds to case (ii), and the middle panel corresponds to case (iii). In each case, a multiple-point or a new fold-point can be created (locally) by arbitrary small perturbations of the nonlinearity $g$ in \eqref{Eq_elliptic-pb}, as described in Theorem \ref{THM_Main}. This results in a topological instability \textemdash\, in the sense of Definition \ref{Def_struct-stab} \textemdash\,  of the one-parameter family of semigroups associated with the corresponding family of parabolic problems.}}   \label{fig:schematic}
%\end{figure}
%%%%%%%%%%

\begin{theorem}\label{THM_Main}
Consider a {\mk locally Lipschitz}, and increasing function
$g:[0,\infty)\rightarrow (0,\infty).$ Let $\Omega$ be a bounded, connected and open
subset of $\mathbb{R}^d$, with either $d=1$ or $d=2$. Let
$\Lambda=[0,\infty)$ and let $\Lambda^*=[0,\lambda^*)$ with
$\lambda^*$ be as defined by Proposition
\ref{Prop_classical-results}. Assume that the solution set
\be
\mathcal{V}_g:=\{(\lambda,\phi)\in [0,\lambda^*)\times
C^{2,\alpha}(\overline{\Omega})\;:\;-\Delta \phi=\lambda g(\phi), \;
\phi|_{\partial \Omega}=0, \; \phi> 0 \mbox{ in } \Omega\},
\ee
 is well defined for some $\alpha \in (0,1)$ and is constituted by a
continuum without multiple-points on it. 

Assume furthermore that the  set of  
fold-points of $\mathcal{V}_g$ given by 
\be
\mathcal{F}:=\{(\lambda,u_{\lambda})\;:\;
(\lambda,u_{\lambda}) \mbox{ is a fold-point of } \mathcal{V}_g\},
\ee
satisfies one of the following conditions
\begin{itemize}
\item[(i)] $\mathcal{F}\neq \emptyset$, $0<\lambda_{\mathfrak{m}}:=\min\{\lambda\in (0,\lambda^*): \mathcal{F}_{\lambda}\neq \emptyset\}<\lambda^*$, and
\bes
\mathcal{V}_g\cap \Gamma^{-}_{\lambda_{\mathfrak{m}}}=\mbox{minimal branch of }  \mathcal{V}_g, 
\ees
where 
\be\label{Eq_Gamma-}
\Gamma_{\lambda_{\mathfrak{m}}}^-=\{(\lambda,\phi)\in
(0,\infty)\times C^{2,\alpha}(\overline{\Omega})\;:\; \lambda
<\lambda_{\mathfrak{m}}, \; \|\phi\|_{\infty} <
\|\underline{u}_{\lambda_{\mathfrak{m}}}\|_{\infty}\}.
\ee
\item[(ii)] $\mathcal{F} \neq \emptyset $ and there exits
$\lambda_{\sharp}\in(0,\lambda^*)$ for which there exists
$\{(\lambda,u_{\lambda})\}_{\lambda \in
(\lambda_{\sharp},\lambda^*)} \subset \mathcal{V}_g$ such that
$$\underset{\lambda\downarrow
\lambda_{\sharp}}\lim\|u_{\lambda}\|_{\infty}=\infty,$$
 with $\mathcal{V}_g\cap
\Gamma^{-}_{\lambda_{\sharp}}$=minimal branch of $\mathcal{V}_g.$
\item[(iii)] $\mathcal{F} =\emptyset $ and $\mathcal{V}_g$ is
constituted only by its minimal branch.
\end{itemize}

{\mkrevb 
One consider now  $\lambda_{\s}$  in $(0,\lambda^*),$ and given $\epsilon>0$, let $\mathcal{P}_{\epsilon}$ be the set of $C^1$-functions $\varphi:[0,\infty)\rightarrow (0,\infty)$ such that 
\be\label{cond1a}%\tag{H$_1$} 
\| \varphi \|_{\infty}
< \epsilon,
\ee
\be\label{cond2b}%\tag{H$_2$} 
 \mbox{supp}(\varphi)\subset (\|\underline{u}_{\lambda_{\s}}\|_{\infty}, \|\underline{u}_{\lambda_{\s}}\|_{\infty}+\epsilon),
 \ee

Let $\mathcal{P}=\cup_{\epsilon>0} \mathcal{P}_{\epsilon}$ and $\mathcal{T}$ be the $C^0$-topology of uniform convergence on compact sets.

Finally,  assume that the family of functions $\mathfrak{F}_g:=\{\lambda g\}_{\lambda \in
[0,\lambda^*)}$ is $(X;C^{2,\alpha}(\overline{\Omega}))$-compatible
relatively to $[0,\lambda^*)$ for some Banach space $X$, and that this family is $\mathcal{T}$-stable with respect to perturbations in $\mathcal{P}$. 

Let $\mathfrak{S}_g$ be the corresponding family of semigroups $\{S_{\lambda}(t)\}_{\lambda\in [0,\lambda^*)}$ associated with 
\bea\label{Eq_parab}
\partial_t u-\Delta u & =\lambda g(u), \; \mbox{ in } \Omega,\\
  u & =0, \quad\quad \mbox{ on } \partial\Omega.\;
\eea

Then $\mathfrak{S}_g$ is topologically unstable  with respect to small perturbations in  $\mathcal{P} $ for the $\mathcal{T}$-topology.

Furthermore, the perturbation $\varphi\in \mathcal{P} $  can be chosen  such that $\widehat{g}=g+\varphi$  is  increasing, 
 for which  $\mathcal{V}_{\widehat{g}}$ contains a multiple-point or a new fold-point compared with $\mathcal{V}_g$, for either $\lambda \in (0,\lambda_{\mathfrak{m}})$,  or $\lambda \in (0,\lambda_{\sharp})$, or $\lambda \in (0,\lambda^{*})$, depending on whether case (i), case (ii), or case (iii), is respectively concerned.

}

\end{theorem}

%%%%%%%%
 % \vspace{-5ex}
\begin{figure}[!hbtp]
   \centering
   \includegraphics[width=.8\textwidth,height=.35\textwidth]{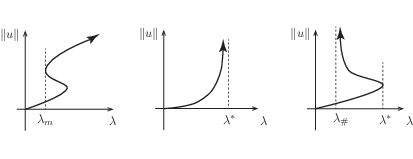}
  \vspace{-5ex}
  \caption{{\footnotesize {\bf Schematic of some typical situations dealt with Theorem \ref{THM_Main}}. The left panel corresponds to case (i), the right panel corresponds to case (ii), and the middle panel corresponds to case (iii). In each case, {\mkrevb either a multiple-point or a new fold-point} can be created (locally) by arbitrary small perturbations of the nonlinearity $g$ in \eqref{Eq_elliptic-pb}, as described in Theorem \ref{THM_Main}. The appearance of such singular points implies a topological instability \textemdash\, in the sense of Definition \ref{Def_struct-stab} \textemdash\,  of the one-parameter family of semigroups associated with the corresponding family of parabolic problems.}}   \label{fig:schematic}
\end{figure}
%%%%%%%%%

\begin{proof}
Let $\mathcal{V}_g$ be the solution set in $[0,\lambda^*)\times
C^{2,\alpha}(\overline{\Omega})$ of \eqref{Eq_elliptic-pb}, {\it
i.e.},
$$\mathcal{V}_g=\{(\lambda,u_{\lambda})\in [0,\lambda^*)\times
C^{2,\alpha}(\overline{\Omega})\;:\; \; -\Delta u_{\lambda}=\lambda
g(u_{\lambda}),u_{\lambda}>0\mbox{ in }
\Omega,\;u_{\lambda}|_{\partial \Omega}=0,\}.$$

First, note that by assumptions  on $\mathfrak{F}_g$,  we have for each $\lambda\in [0,\lambda^*)$ the existence of $D(\lambda g) \subset X$  such that 
 Eq. ~\eqref{Eq_parab} generates a semigroup acting on  $D(\lambda g)$; see Definition \ref{Def_Admissible}.
By introducing $\widetilde{D}(\lambda g)=D(\lambda g)\cap \{\phi > 0\mbox{ in } \Omega\}$, we can still
define a semigroup $\{S_{\lambda}(t)\}_{t \geq 0}$ acting on
$\widetilde{D}(\lambda g),$ due to the maximum principle.

Let us recall now {\mkrevb the implications of \cite[Theorem 1.2]{Nao}.  The latter theorem takes place in dimension one or two. It ensures  the existence of a locally Lipschitz, positive and increasing function $\widehat{g}$ that can be chosen arbitrarily close to $g$ in the {\mkrevb $C^0$-topology of uniform convergence on compact sets,} and for which  the branch of
minimal positive solutions, $\lambda \mapsto \underline{\widehat{u}}_{\lambda}$, of } 
\begin{equation}\label{Eq_Perturb-pb}
\left\{
\begin{array}{l}
-\Delta u =\lambda  \widehat{g}(u),\;  \;  \; u> 0 \; \mbox{ in } \Omega,\\
\hspace{1em}u|_{\partial \Omega}=0,
\end{array}
\right.
\end{equation}
undergoes a discontinuity of first kind, as a map from $(0,\widehat{\lambda}^*)$
to $ C^{2,\alpha}(\overline{\Omega})$.\footnote{In \cite{Nao} the
authors have proved the existence of such a discontinuity in the
$L^{\infty}(\Omega)$-norm for solutions considered in
$C^2(\overline{\Omega})$ which is therefore valid {\mk for solutions considered in $
C^{2,\alpha}(\overline{\Omega})$}. {\mkrevb Their proof has been also done for $C^1$ functions $g$, but can be adapted to the case of locally  Lipschitz functions  since only the monotony property of the minimal branch is needed from that assumption; see also Remark \ref{Rem_regul_growth}.}}

More precisely, let $\lambda_{\s}$ be chosen in $(0,\lambda^*).$ Given $\epsilon >0$,
 \cite[Theorem 1.2]{Nao} ensures the  existence of  {\mkr an increasing {\mkrevb locally Lipschitz}
positive function $\widehat{g}$, such that the following conditions hold:}
\be\label{cond1}\tag{H$_1$} 
\|g-\widehat{g}\|_{\infty}
\leq \epsilon,
\ee
\be\label{cond2}\tag{H$_2$} 
 \mbox{supp}(g-\widehat{g})\subset (\|\underline{u}_{\lambda_{\s}}\|_{\infty}, \|\underline{u}_{\lambda_{\s}}\|_{\infty}+\epsilon),
 \ee
for which  the following set 
 \bes
 \mathcal{M}=\{\underline{\widehat{u}}_{\lambda}, \; \lambda\in \widehat{\Lambda}^{*}\},
 \ees
is constituted by minimal solutions of \eqref{Eq_Perturb-pb} over an interval $ \widehat{\Lambda}^{*}:=(0, \widehat{\lambda}^*)$ such that 
 \be \label{cond3}\tag{H$_3$} 
 \widehat{\lambda}^* > \lambda_{\s}, \; 
\underline{\widehat{u}}_{\lambda}=\underline{u}_{\lambda}  \mbox{ for }
\lambda\in(0,\lambda_{\s}), \mbox{ and }   \lambda \mapsto \underline{\widehat{u}}_{\lambda} \mbox{ is discontinuous on }
(\lambda_{\s}, \lambda_{\s}+\epsilon).
\ee

{\mkrev Conditions \eqref{cond1}-\eqref{cond2} indicate that the  perturbation $\widehat{g}(x)$ of $g(x)$ is localized for the $x$-values located near $\|\underline{u}_{\lambda_{\s}}\|_{\infty}$ for some $\lambda_{\s}$, and Condition \eqref{cond3} expresses that such a perturbation  generates a discontinuity near $\lambda_{\s}$ on the minimal branch associated with \eqref{Eq_Perturb-pb}.}

{\bf Case (i)}. We consider
\bes
\mathcal{F}=\{(\lambda,u_{\lambda})\;:\;
(\lambda,u_{\lambda}) \mbox{ is a fold-point of } \mathcal{V}_g\},
\ees
and assume first that $\mathcal{F}\neq \emptyset$ and that the
condition (i) such as formulated  in the statement of the theorem, is satisfied.

Let us choose $\epsilon>0$
and $\lambda_{\s}$ such that, 
\be\label{Eq_lambda_m}
0<{\mkr \lambda_{\s}+2\epsilon}\leq\lambda_{\mathfrak{m}}:=
\min\{\lambda\;:\;(\lambda,u_{\lambda})\in \mathcal{F}\}
\ee
and such that
\be \label{Eq_lambda_m2}
 \|\underline{u}_{\lambda_{\s}}\|_{\infty}+\epsilon <\|\underline{u}_{\lambda_{\mathfrak{m}}}\|_{\infty}. 
 \ee
The latter is possible by monotony of the minimal branch; see Proposition \ref{Prop_classical-results}.

For this choice of 
$\lambda_{\s}$ and $\epsilon$, and for the corresponding perturbation
$\widehat{g}$ of $g$ verifying Conditions \eqref{cond1}-\eqref{cond3},  similar topological degree arguments  (Theorem
\ref{THM_global_unbounded}) to those provided for the  Gelfand problem  \eqref{Eq_Gelfand-perturb} in Section 
\ref{Sec_perturbed-Gelfand}, ensure the existence of unbounded continuum in $\widehat{\Lambda}^* \times V$, with  
here $V=C^{2,\alpha}(\overline{\Omega})$.

Let $\lambda_{\c}\in (\lambda_{\s}, \lambda_{\s}+\epsilon)$ be the {\mkr critical}
parameter value at which the {\it discontinuity} of the minimal branch, $\lambda
\mapsto\underline{\widehat{u}}_{\lambda}$, takes place. Let $\widehat{\mathcal{C}}$ be  the
unbounded continuum of $\mathcal{V}_{\widehat{g}}$ which contains
$(0,0_V)$.  By construction of $\widehat{g}$ and assumption on
$\mathcal{V}_g$,  we deduce that 
\be\label{Eq_multiple}
\widehat{\mathcal{C}}\cap
\Gamma_{\lambda_{\s}}^-=\{(\lambda,\underline{u}_{\lambda})\}_{\lambda <
\lambda_{\s}},
\ee 
where $\Gamma_{\lambda_{\s}}^-$  is defined as in Eq.~\eqref{Eq_Gamma-}, by replacing $\lambda_{\mathfrak{m}}$ with 
$\lambda_{\s}$. Hereafter, we define similarly the set $\Gamma_{\lambda_{\c}}^-$.

Assume first that,
\bes
\{(\lambda,\underline{\widehat{u}}_{\lambda})\}_{\lambda <
\lambda_{\c}}\varsubsetneq\widehat{\mathcal{C}}\cap
\Gamma_{\lambda_{\c}}^-.
\ees
{\mkrevb Then because of \eqref{Eq_multiple} and the definition of $\Gamma_{\lambda_{\c}}^-$, the solution set $\widehat{\mathcal{C}}$  contains solutions $\phi_{\lambda}$ of Eq.~\eqref{Eq_Perturb-pb} such that  $\|\phi\|_{\infty}<\|\underline{\widehat{u}}_{\lambda}\|_{\infty}$ for  $\lambda_{\s}\leq \lambda <
\lambda_{\c}$. Given the continuum property of $\widehat{\mathcal{C}}$, such a subset of solutions form a branch that necessarily intercepts the set}
\bes
\{(\lambda,\underline{\widehat{u}}_{\lambda})\}_{\lambda_{\s}\leq \lambda <
\lambda_{\c}},
\ees 
at some point
$(\lambda,\underline{\widehat{u}}_{\lambda})$ for $\lambda \in
[\lambda_{\s}, \lambda_{\c})$, leading to the existence of a {\it multiple-point} of $\mathcal{V}_{\widehat{g}}$ which turns out to be a signature
of topological instability of $\mathfrak{S}_g$ {\mkr according  to Proposition \ref{Prop_criteria-noneq}-(iii) and  to the assumption made on
$\mathcal{V}_g$.}

{\mkr Consider  now the case where}
\be
\{(\lambda,\underline{\widehat{u}}_{\lambda})\}_{\lambda <
\lambda_{\c}}=\widehat{\mathcal{C}}\cap \Gamma_{\lambda_{\c}}^-.
\ee
{\mkr A more careful analysis is here required to conclude to the topological instability of $\mathfrak{S}_g$.}

{\mkr First, let us note that standard 
compactness arguments allow us to conclude to the existence of  a 
sequence $\{\lambda_k\}$, such that} 
\bes
v_{\lambda_{\c}}:=\underset{\lambda_k\uparrow\lambda_{\c}}\lim\underline{\widehat{u}}_{\lambda_k} \mbox{ exists},
\ees
and {\mkr such} that this limit is a solution of
\eqref{Eq_Perturb-pb} for $\lambda=\lambda_{\c}$.

{\mkr This solution} has to be the minimal solution at
$\lambda_{\c}$ since from the construction of \cite{Nao}, we deduce
\be\label{Eq_control_left-right}
\underset{\lambda \uparrow \lambda_{\c}}\lim\;
\|\underline{\widehat{u}}_{\lambda}\|_{\infty}< \underset{\lambda
\downarrow \lambda_{\c}}\lim\;
\|\underline{\widehat{u}}_{\lambda}\|_{\infty}. 
\ee

Therefore,
\be
v_{\lambda_{\c}}=\underline{\widehat{u}}_{\lambda_{\c}}  \mbox{ and }
(\lambda_{\c},\underline{\widehat{u}}_{\lambda_{\c}})\in
\widehat{\mathcal{C}}.
\ee

Denote by $A_{\lambda_{\c}}^+$ the point
$(\lambda_{\c},\underset{\lambda \downarrow
\lambda_{\c}}\lim\; \underline{\widehat{u}}_{\lambda})$ which exists
from same arguments of compactness.  Similarly, we get that
$A_{\lambda_{\c}}^+=(\lambda_{\c},
\widehat{u}_{\lambda_{\c}}^+)$ for some
$\widehat{u}_{\lambda_{\c}}^+\in \mathcal{V}_{\widehat{g}}$.

Since $\widehat{u}_{\lambda_{\c}}^+=\underset{\lambda \downarrow
\lambda_{\c}}\lim\; \underline{\widehat{u}}_{\lambda}$, and 
$\lambda_{\c}<\lambda_{\mathfrak{m}}$ by construction, and since the map
$\lambda\mapsto \underline{\widehat{u}}_{\lambda}$ is increasing from
Proposition \ref{Prop_classical-results}-(ii), we infer that
necessarily,
\be\label{Eq_domination}
\|\widehat{u}_{\lambda_{\c}}^+\|_{\infty}<\|\underline{u}_{\lambda_{\mathfrak{m}}}\|_{\infty}.
\ee
{\mkrev ln other words,  the right-hand limit at the critical parameter value $\lambda_{\c}$ of the minimal solutions to the perturbed problem \eqref{Eq_Perturb-pb}, comes with less energy than the energy of the first fold-point\footnote{i.e.~the first fold-point met as $\lambda$ is increased from $0$.}  associated with the unperturbed problem \eqref{Eq_elliptic-pb}.}

 Since $\widehat{\mathcal{C}}$ is unbounded in
$\Lambda\times V,$ either 
$(\lambda_{\c},\underline{\widehat{u}}_{\lambda_{\c}})$ is a
fold-point of $\widehat{\mathcal{C}}$ that lives thus according to \eqref{Eq_domination} in $\Gamma_{\lambda_{\mathfrak{m}}}^-$, or 
$(\lambda_{\c},\underline{\widehat{u}}_{\lambda_{\c}})$ is not a
fold-point of $\widehat{\mathcal{C}}$ and $\widehat{\mathcal{C}} \cap \Gamma_{\lambda_{\c},\gamma}^+\neq
\emptyset$ for all $\gamma>0$, where
\bes
\Gamma_{\lambda_{\c},\gamma}^+:=\{(\lambda,v)\in \Lambda \times
V\;:\; \lambda>\lambda_{\c}, \;
\|v-\underline{\widehat{u}}_{\lambda_{\c}}\|_V < \gamma\}. 
\ees

Let us show that the second option of this alternative does not hold. 
By contradiction, assume that $\widehat{\mathcal{C}} \cap
\Gamma_{\lambda_{\c}}^+\neq \emptyset$ for all $\gamma>0$ and that 
$(\lambda_{\c},\underline{\widehat{u}}_{\lambda_{\c}})$ is not a
fold-point of $\widehat{\mathcal{C}}$, then condition (F$_2$) of Definition
\ref{Def_Fold-point} is violated and therefore any local continuous
map given for some $\theta>0$ as,
\bes
\mu:s\in (-\theta,\theta) \mapsto (\lambda(s),v(s)), 
\ees
and such that for all $s\in (-\theta,\theta)$, 
$(\lambda(s),v(s)) \in \widehat{C}$ with
$(\lambda(0),v(0))=(\lambda_{\c},\underline{\widehat{u}}_{\lambda_{\c}})$,
comes with its underlying map 
\bes 
s\mapsto\lambda(s),
\ees
that does not attain its maximum  at 
$s=0$.

{\mkr Recall from Eq.~\eqref{Eq_control_left-right} that
\be
\|\widehat{u}_{\lambda_{\c}}\|_{\infty}<\|\widehat{u}_{\lambda_{\c}}^+\|_{\infty}.
\ee
}
Then by continuity of the map $\mu$ there
exists $0<\beta\leq \theta$ such that $s\mapsto\lambda(s)$ is
strictly increasing on $(0,\beta)$ and such that 
\be
\|v(s)\|_{\infty}<\|\widehat{u}_{\lambda_{\c}}^+\|_{\infty}, \; \forall\, s\in(0,\beta).
\ee

This last inequality  is in contradiction with the minimality property of the
branch $\lambda\mapsto \underline{\widehat{u}}_{\lambda}$ and the fact that,  by construction of $\widehat{u}_{\lambda_{\c}}^+$, 
$\|\underline{\widehat{u}}_{\lambda}\|_{\infty}\geq
\|\widehat{u}_{\lambda_{\c}}^+\|_{\infty}$ for any $\lambda>
\lambda_{\c}$ such that $\lambda-\lambda_{\c}$ is small
enough. 

{\mkr Thus, the second part of the aforementioned  alternative does not hold which implies that} $(\lambda_{\c},\underline{\widehat{u}}_{\lambda_{\c}})$ is a
fold-point of $\widehat{\mathcal{C}}$ that lives  according to \eqref{Eq_domination} in  $\Gamma_{\lambda_{\mathfrak{m}}}^-$. 
By definition of $\lambda_{\mathfrak{m}}$ in \eqref{Eq_lambda_m}, no fold-point exists in $\Gamma_{\lambda_{\mathfrak{m}}}^-$ for $ \mathcal{V}_g$. On the other hand, recall that by construction of $\widehat{g}$ satisfying  \eqref{cond1}-\eqref{cond3} for $\epsilon$ and $\lambda_{\s}$ satisfying \eqref{Eq_lambda_m}-\eqref{Eq_lambda_m2}, one has  that $g(x)=\widehat{g}(x)$ for $x>\|\underline{u}_{\lambda_{\mathfrak{m}}}\|_{\infty}$ and hence
\be
\mathcal{V}_{\widehat{g}}\cap
\Gamma_{\lambda_{\mathfrak{m}}}^+=\mathcal{V}_{g}\cap \Gamma_{\lambda_{\mathfrak{m}}}^+,
\ee
where
\be
\Gamma_{\lambda_{\mathfrak{m}}}^+:=\{(\lambda,\phi)\in
(0,\infty)\times C^{2,\alpha}(\overline{\Omega})\;:\; \lambda
>\lambda_{\mathfrak{m}}, \; \|\phi\|_{\infty} >
\|\underline{u}_{\lambda_{\mathfrak{m}}}\|_{\infty}\}.
\ee

{\mkr As a consequence,} the set of fold-points in $\Gamma_{\lambda_{\mathfrak{m}}}^+$ of
$\mathcal{V}_{\widehat{g}}$ and $\mathcal{V}_{g}$ are identical. We
have just proved the existence of a fold-point of
$\mathcal{V}_{\widehat{g}}$ in $(0,\lambda_{\mathfrak{m}})\times X$ which no longer
exists \textemdash\, in an homeomorphic sense \textemdash\, on $\mathcal{V}_{g}$ by
definition of $\lambda_{\mathfrak{m}}$. From Proposition
\ref{Prop_criteria-noneq}-(i), we conclude that $\mathfrak{S}_g$ and
$\mathfrak{S}_{\widehat{g}}$ are thus not topologically equivalent.

{\bf Case (ii)}. The proof follows the same lines than above by working with $(0,\lambda_{\sharp})$ instead of  $(0,\lambda_{\mathfrak{m}})$, and by localizing the perturbation on $\widehat{\mathcal{C}}\cap
\Gamma^-_{\lambda_{\sharp}}$.

{\bf Case (iii)}. If $\mathcal{F}=\emptyset$, $\lambda_{\s}$ may be
chosen arbitrary in $(0,\lambda^*)$, and we can proceed as above
to create a fold-point of $\mathcal{V}_{\widehat{g}}$ whereas
$\mathcal{V}_{g}$ does not possess any fold-point
($\mathcal{F}=\emptyset$). 
\medskip 

In all the cases, we are thus able to exhibit for any 
$\epsilon>0$, a perturbation $\widehat{g}$ for which 
$\|g-\widehat{g}\|_{\infty} \leq \epsilon$ while  $\mathfrak{S}_g$ and
$\mathfrak{S}_{\widehat{g}}$ are not topologically equivalent. We have thus proved that $\mathfrak{S}_g$ is
{\mkk topologically} unstable in the sense of Definition
\ref{Def_struct-stab}. The proof is complete.
\end{proof}

%%%%

\begin{Rmk}\label{Rmk_linearization}
If one assumes $g$ to be $C^1$ {\mkrevb instead of locally
Liptchitz, and assumes also  
$(\lambda_{\c},\underline{\widehat{u}}_{\lambda_{\c}})$ used in the proof above},  
to be degenerate in the sense that
\bes
\lambda_1(-\Delta -\lambda_{\c}g'(\underline{\widehat{u}}_{\lambda_{\c}})I)=0, 
\ees
and the linearized equation has a nontrivial solution,  then under further assumptions on $g$ and  appropriate {\it a priori} bounds,  the
existence of a fold-point at $(\lambda_{\c},\underline{\widehat{u}}_{\lambda_{\c}})$ can be guaranteed by using e.g.~ \cite[Theorem 1.1]{Crandall_Rabinowitz'75}; see also \cite{Crandall_Rabinowitz'73,OS99}.

The regularity
assumption on $g$ in Theorem \ref{THM_Main} prevents the use of such linearization techniques. 
Note that parabolic problems with locally Lipschitz nonlinearities are commonly encountered in energy balance models \cite{Roques_al14} and in some population dynamics models \cite{CPT16,RC07}. 

 Theorem \ref{THM_global_unbounded}  serves here as  a substitutive ingredient to cope with the lack of regularity caused by our assumptions on $g$.  It is however unclear how to weaken further these assumptions, since the proof of Theorem \ref{THM_Main} provided above  has made a substantial use of the growth property of the minimal branch such as recalled in Proposition \ref{Prop_classical-results} above; see also Remark \ref{Rem_regul_growth}.
\end{Rmk}

%%%%%%%%%%%%%%%%%%
{\mkrevb We conclude this section by an application  to the parabolic version of the  perturbed Gelfand problem \eqref{Eq_Gelfand-perturb} discussed  in Section \ref{Sec_topo_rob_mother}. }

{\mkrevb 

\begin{Cor}\label{Main_cor}
Let  $\lambda>0$ and $\epsilon >1/4$.

Let $\mathfrak{S}_g$ be the family of semigroups $\{S_{\lambda}(t)\}_{\lambda>0}$ defined on $D(f_{\lambda})$ given by \eqref{Domain_Semi-group}, associated with 
\begin{equation}\label{Eq_Gelfand-perturb2}
\left\{
\begin{array}{l}
\displaystyle  \partial_t u -\partial_{xx}^2 u = \lambda \exp\Big(\frac{u}{1+\epsilon u}\Big), \;\textrm{ in }I=(-1,1),\\
\hspace{1.8em} u(-1) =u(1)=0, 
\end{array}\right.
\end{equation}

Let $\mathcal{T}$ and  $\mathcal{P}$  be as in Theorem \ref{THM_Main}.

Then $\mathfrak{S}_g$ is  topologically unstable with respect to small perturbations in  $\mathcal{P} $ for the $\mathcal{T}$-topology.\end{Cor}

\begin{proof}
 Let us consider $f_{\lambda}=\lambda g$, with $g(x)=\exp(x/(1+\epsilon x))$,  and   $\lambda\in \Lambda=(0,\infty)$.
From Proposition \ref{Prop_classical-results}, $\lambda^\ast=\infty$ and therefore $\Lambda^\ast=(0,\infty).$

From Example \ref{Rem_Arrenus}, we know that   $\mathfrak{F}=\{f_{\lambda}\}_{\lambda \in\Lambda}$ is $(C_0^{0,2\alpha}(I);C^{2}(I))$-compatible relatively to $\Lambda$ for $\alpha \in (\frac{1}{2},1)$ and $\epsilon >1/4$.  

Let $\alpha$ be fixed in $(\frac{1}{2},1)$ and $\epsilon>1/4.$  By application of Proposition \ref{Prop_identif_admissible},  we know also that case (iii) of Theorem \ref{THM_Main}  holds here.
It remains to check that $\mathfrak{F}$ is $\mathcal{T}$-stable with respect to perturbations in $\mathcal{P}$, namely that 
the family 
$\{\lambda (g+\varphi)\}_{\lambda \in \Lambda'}$ is 
$(C_0^{0,2\alpha}(I);C^{2}(I))$-compatible relatively to $\Lambda'=(0,\infty)$, for $\varphi \in \mathcal{P}$ sufficiently small. 

Since $\varphi$ is C$^1$ and with compact support, there exists
$C>0$ such that   $g(x)+\varphi(x) \leq C (1+x)$, for all $x\geq 0$. 
 The theory of analytic semigroups guarantees then the existence  of a semigroup $\widehat{S}_{\lambda}(t)$ defined, for each $\lambda>0$, on
\begin{equation}\label{Domain_Semi-groupb}
D(\lambda(g+\varphi)):=\{u_0\in   C_0^{0,2\alpha}(I)\;:\;
\underset{t>0}\sup\;\|\widehat{u}_{\lambda}(t;u_{0})\|_{C^{2}(I)}<\infty\},
\end{equation}
where $\widehat{u}_\lambda(t;u_0)$ denotes the unique solution of $\partial_t u-\partial_{xx}^2 u= \lambda (g(u) +\varphi(u))$, with $u(-1)=u(1)=0$, and emanating from $u_0\in C_0^{0,2\alpha}(I)$; see e.g.~\cite[Props.~
6.3.5. and 6.3.8]{Lunardi_Lecture-notes}.  Thus, Condition (i) of Definition \ref{Def_Admissible} is satisfied for $g+\varphi$. 

From the assumptions on $\varphi$, the method of super- and subsolutions (see e.g.~\cite[Chap.~3]{Nao_book}) allows us to show that   Condition (ii) of Definition \ref{Def_Admissible} is satisfied for $g+\varphi$. Indeed, since $\varphi\geq 0$, any solution of $-\partial_{xx}^2 u = \lambda g(u)$  (under Dirichlet conditions) produces  a subsolution $\underline{v}$ (in $C^{2}(I)$) of \eqref{Eq_Perturb-pb} with $\widehat{g}=g+\varphi$. Recall now that the minimal branch of \eqref{Eq_Gelfand-perturb} is an increasing function of $\lambda$ (see Proposition \ref{Prop_classical-results} (ii))  that coincides with the the solution set of \eqref{Eq_Gelfand-perturb} for $\epsilon>1/4$. As a consequence,  given $\lambda>0$, any solution of $-\partial_{xx}^2 u = (\lambda +\gamma) g(u)$ for $\gamma$ sufficiently large provides  a supersolution $\overline{v}$ of \eqref{Eq_Perturb-pb}   for which $\overline{v}\geq \underline{v}$. The existence of a solution to \eqref{Eq_Perturb-pb} with $\widehat{g}=g+\varphi$ follows then from a classical iteration method.

Finally, any solution in $C^{2}(I)$ of  $-\partial_{xx}^2 u = \lambda (g+\varphi)(u)$, under Dirichlet conditions, is clearly an equilibrium of $\widehat{S}_{\lambda}$.   The perturbation $\varphi$ being allowed to be arbitrarily small in $\mathcal{T}$, we have thus proved that $\mathfrak{F}$ is $\mathcal{T}$-stable with respect to perturbations in $\mathcal{P}$. The application of Theorem \ref{THM_Main}  concludes the proof.

\end{proof}

}

\section{Numerical results}\label{Sec_num}
In this section we complete the theoretical results of Section \ref{Sec_Main} by numerical simulations.
We consider the following Gelfand problem
\begin{equation}\label{Eq_Gelfand-perturb3}
\left\{
\begin{array}{l}
\displaystyle  \partial_t u -\nu\partial_{xx}^2 u = \lambda \exp\Big(\frac{u}{1+\epsilon u}\Big)=\lambda g(u), \;\textrm{ in }I=(0,1),\\
\hspace{1.8em} u(0) =u(1)=0,
\end{array}\right.
\end{equation}
with $\nu=0.01$ and $\epsilon=0.4$.

The nonlinearity $g$ is subject to the following small Gaussian perturbations of the form 
\be\label{perturb_Gaussian}
\varphi(y)=\epsilon_1 \exp\Big(-\frac{\beta}{\epsilon_1}\big(y-\|u_{\lambda_\s}\|_{\infty}\big)^2\Big),
\ee
with $\epsilon_1=0.75$ and $\beta=20$, and where $u_{\lambda_s}$ denotes the (unique) stationary solution of \eqref{Eq_Gelfand-perturb3} for $\lambda=\lambda_s=0.11$. Note that $\|\varphi\|_{\infty}\leq \eps_1.$

The goal is to numerically illustrate that the perturbed problem
\begin{equation}\label{Eq_Gelfand-perturb4}
\left\{
\begin{array}{l}
\displaystyle  \partial_t u -\nu\partial_{xx}^2 u =\lambda (g(u)+\varphi(u))=\lambda \widehat{g}(u), \;\textrm{ in }I=(0,1),\\
\hspace{1.8em} u(0) =u(1)=0,
\end{array}\right.
\end{equation}
is topologically non-equivalent to \eqref{Eq_Gelfand-perturb3}. Since the perturbation $\varphi$ given by \eqref{perturb_Gaussian} does not fall within the set of perturbations $\mathcal{P}$ considered in Theorem \ref{THM_Main}, the numerical results shown hereafter strongly suggest  that the topological instability of problems such as \eqref{Eq_Gelfand-perturb3} is not limited to perturbations in $\mathcal{P}$.

The (locally) stable stationary solutions of  either \eqref{Eq_Gelfand-perturb3} or \eqref{Eq_Gelfand-perturb4} are approximated from a standard explicit finite differentiation with a number of grid points  sets to $N_x=100$, and a time increment sets to $\delta t=10^{-3}$. A total of $10^5$ iterations has been used.
For  either \eqref{Eq_Gelfand-perturb3} or \eqref{Eq_Gelfand-perturb4}, the computation of the minimal branch  is obtained by integration from the following square wave function
\be
u_0(x)=\left\{
\begin{array}{l}
0.5, \mbox{ if } x\in[\frac{1}{4},\frac{3}{4}],\\
0,  \mbox{ else. }
\end{array}\right.
\ee

In both cases, $\lambda$ runs from $\lambda_1=0.01$ to $\lambda_2=0.2$ with increment $\delta \lambda=5.10^{-5}$.
For each $\lambda$, the upper branch of stationary solutions of the perturbed branch (red curve on Fig.~\ref{fig:numerics}) is obtained by integration of 
\eqref{Eq_Gelfand-perturb4} from 
\be
u_0(x)=\left\{
\begin{array}{l}
\|u_{\lambda}\|_{\infty}+0.1, \mbox{ if } x\in[0.2,0.8],\\
0,  \mbox{ else, }
\end{array}\right.
\ee
where $u_{\lambda}$ denotes the stationary solution of  \eqref{Eq_Gelfand-perturb3}. A standard method of continuation 
has been used for computing the unstable branch.

\begin{figure}[!hbtp]
   \centering
   \includegraphics[width=.8\textwidth,height=.5\textwidth]{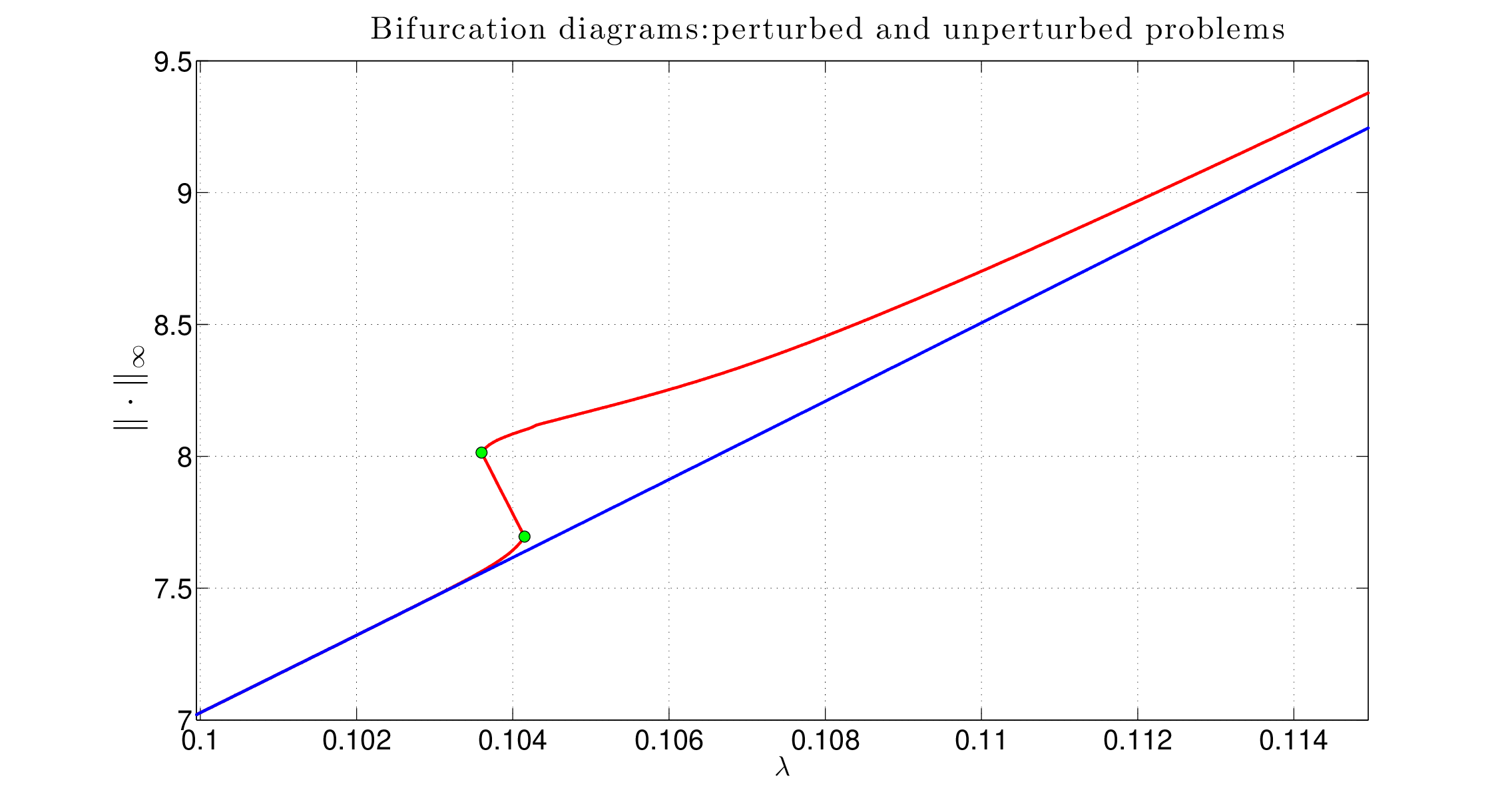}
%  \vspace{-5ex}
  \caption{{\footnotesize  Bifurcation diagrams for the perturbed problem (red curve) and the unperturbed one (blue curve). The fold-points are indicated by the green dots.}   }
  \label{fig:numerics}
\end{figure}
%%%%%%%%%

The results are shown in Fig.~\ref{fig:numerics}.  Compared to the set of stationary solutions associated with   \eqref{Eq_Gelfand-perturb3} (blue curve), the set of stationary solutions associated with   \eqref{Eq_Gelfand-perturb4} (red curve) exhibits two fold-points (green dots). Figure \ref{fig:numerics}  represents actually a magnification of the discrepancies between these two solution sets. It has indeed been observed that the distance between the red and blue curves decays to zero  (not shown) as $\lambda$ gets larger from its critical value $\lambda_{\c}$ at which a discontinuity of the minimal branch occurs.

It is interesting to remark that $\lambda_{\c}$  is here slightly bigger than $\lambda=0.104$ but smaller than $\lambda_{\s}=0.11$, contrarily to Property \eqref{cond3} satisfied for a perturbation in $\mathcal{P}$ from Theorem \ref{THM_Main}. 
Here  $\lambda_{\s}$ corresponds to the parameter value from which the Gaussian perturbation $\varphi$ has been centered via $\|u_{\lambda_\s}\|_{\infty}$, whereas for a perturbation in $\mathcal{P}$, $\|u_{\lambda_\s}\|_{\infty}$  corresponds to a lower bound of the support of the  perturbation.

 It has been finally observed numerically that the emergence of fold-points such as reported on Fig.~\ref{fig:numerics}, persists when the perturbation $\varphi$ from \eqref{perturb_Gaussian} is still employed while $\eps_1>0$ is further reduced. 
The rigorous justification of this observation boils down again essentially  to an understanding of the mechanism at the origin of a discontinuity in the minimal branch, when this time a perturbation such as given in \eqref{perturb_Gaussian} is applied. We leave this issue for a future research, pointing out in the concluding remarks below a key element for the creation of such a discontinuity from the perturbation techniques of \cite{Nao}.

 \section{Concluding remarks}\label{concl_rem}
 The creation of a discontinuity in the minimal branch by arbitrarily small perturbations of the nonlinearity, has played a crucial role in the proof of Theorem \ref{THM_Main}. 
 This is made possible when the spatial dimension is equal to one or two, due to the following observation regarding a specific Poisson equation used in the {\mkrevb perturbation techniques of  \cite{Nao}.}

 Given $r>0$, one denote by $B_r$ the open ball of $\mathbb{R}^d$ of radius $r$, centered at the origin.  For $0<\rho<R$, the solution $\Psi_{\rho}$ of the following Poisson equation 
\begin{equation}\label{Eq_Poisson}
 \left\{
\begin{array}{l}
-\Delta \Psi_{\rho} =1_{B_{\rho}} ,  \textrm{ in } B_R,\\
\; \Psi_{\rho} |_{\partial B_R}=0, 
 \end{array}
 \right.
\end{equation}
satisfies  for $\rho <R/2$,
\be
\underset{B_{2\rho}}\inf \Psi_{\rho}= \rho^2 K(\rho),
\ee
where the behavior of  $K(\rho)$ as $\rho \rightarrow 0$ is of the form
\be
 K(\rho)\approx \left\{ 
 \begin{array}{l} 
 R/\rho, \hspace{3.35em} \textrm{ if } d=1,\\
 |\log \rho |/2, \hspace{1.5em} \textrm{ if } d=2.
  \end{array}
 \right.
\ee
This asymptotic behavior of $K(\rho)$ near $0$ can be proved  by  simply writing down the analytic expression of the solution to  \eqref{Eq_Poisson}; see \cite[Lemma 3.1]{Nao}.

When $d\geq 3$,  $K(\rho)$ converges to a  constant (depending on $d$) as $\rho \rightarrow 0.$
This removal of the singularity at $0$ for $K$ in dimension $d\geq 3$, implies that the perturbation constructed from the techniques of \cite{Nao} needs to be sufficiently large to generate a discontinuity in the minimal branch.  Whether this point is purely technical or more substantial, is still an open problem.

\appendix

\section{Unbounded continuum of solutions to parametrized fixed point problems, in Banach spaces}\label{Sec_AppendixA}
We communicate in this appendix on a general
result concerning the existence of an unbounded continuum of fixed
points associated with one-parameter families of completely continuous
perturbations of the identity map in a Banach space. This theorem is rooted in the seminal work of \cite{LS34} that initiated what is known today as the {\it Leray-Schauder continuation theorem}. 
Extensions of such a continuation result  can be found in \cite{FMP86,MP84} for the multi-parameter case.   Theorem \ref{THM_global_unbounded} below, formulates  such a result in the one-parameter case.  Its proof is provided here to make the expository as much self-contained as possible.  Under a nonzero condition on the Leray-Schauder  degree to hold at some parameter value, Theorem \ref{THM_global_unbounded} ensures in particular  the existence of an unbounded continuum of solutions to nonlinear eigenvalue problems for which the nonlinearity is not necessarily Fr\'echet differentiable.

Results similar to Theorem \ref{THM_global_unbounded} that deal with the existence of an unbounded continuum of solutions to nonlinear eigenvalue problems, have been obtained in the literature, see e.g.~
\cite[Theorem 3.2]{rab2}, ~\cite[Corollary 1.34]{Rab73_continua}, \cite[Theorem 3]{Brezis_Beres} or \cite[Theorem 17.1]{Amann}.  Similar to these works, the ingredients for proving Theorem \ref{THM_global_unbounded} rely also on the Leray-Schauder degree properties and {\mkk connectivity  arguments} from point set topology. However, by following the approach  of \cite{FMP86,MP84}, Theorem \ref{THM_global_unbounded} ensures the existence of an unbounded continuum of solutions to parameterized fixed point problems under more general conditions on the nonlinear term  than required in \cite{rab2,Rab73_continua,Brezis_Beres,Amann}.

Hereafter, given  a real Banach space $E$ and a map $\Psi:E\rightarrow E$, $\deg (\Psi, \mathcal{O},y)$ stands for the 
Leray-Schauder degree of $\Psi$ with respect to an open bounded subset $\mathcal{O}$ of $E$, and $y\in E.$ This degree is well defined for completely continuous perturbations $\Psi$ of
the identity map and  if $y\not\in \Psi( \partial
\mathcal{O})$;
see e.g. ~\cite[Chap.~2,Thm.~8.1]{Deimling}. In what follows the  $\lambda$-section of a nonempty subset $\mathcal{A}$ of $\mathbb{R}_{+}\times E $, is defined as:
\be
\mathcal{A}_{\lambda}:=\{u\in E\;:\; (\lambda,u)\in
\mathcal{A}\}.
\ee

\begin{theorem}\label{THM_global_unbounded}
Let $\mathcal{U}$ be an open bounded
subset of a real Banach space $E$ and assume that $G:
\mathbb{R}_{+}\times E \rightarrow E$ is completely continuous ({\it
i.e.} ~compact and continuous). We assume that there exists
$\lambda_0\geq 0$, such that the equation,
\begin{equation}\label{Eq_Nonlinear-eq}
u-G( \lambda_0,u)=0
\end{equation}
has a unique solution $u_0$, and,
\begin{equation}\label{Eq_deg-non-zero}
\deg (I-G(\lambda_0, \cdot), \mathcal{U}, 0)\neq 0.
\end{equation}

Let us introduce  
\be
\mathcal{S}^+=\{(\lambda,u)\in [\lambda_0,\infty)\times
E: u=G(\lambda,u)\}.
\ee

Then there exists a continuum
$\mathcal{C}^{+} \subseteq \mathcal{S}^{+}$ (i.e.~a closed and
connected subset of $\mathcal{S}^{+}$) such that the following
properties hold:
\begin{itemize}
  \item[(i)] $\mathcal{C}_{\lambda_0}^{+}\cap \mathcal{U}=\{u_0\}$, 
  \item [(ii)] Either $\mathcal{C}^{+}$ is unbounded or $\mathcal{C}_{\lambda_0}^{+}\cap
  (E\setminus \overline{\mathcal{U}}))\neq \emptyset.$
\end{itemize}
\end{theorem}

To prove this theorem, we need an extension of  the
 standard homotopy property of the Leray-Schauder degree \cite[p.~56]{Deimling} to homotopy
cylinders that exhibit variable $\lambda$-sections.  This is the purpose of the following Lemma. 
\begin{Lem}\label{Lem_Homotopy_gene}
Let $\mathcal{O}$ be a bounded open subset of $[\lambda_1,\lambda_2]\times E$, and let $G:\overline{\mathcal{O}}\rightarrow
E$ be a completely continuous mapping. Assume that $u\neq
G(\lambda,u)$ on $\partial \mathcal{O},$ then for all $\lambda \in
[\lambda_1,\lambda_2]$,
$$\deg (I-G(\lambda,\cdot),\mathcal{O}_{\lambda},0_E) \mbox{ is independent of } \lambda,$$
where $\mathcal{O}_{\lambda}=\{u\in E\;:\; (\lambda,u)\in
\mathcal{O}\}$ is the $\lambda$-section of $\mathcal{O}$.
\end{Lem}

\begin{proof}
We may assume, without loss of generality, that $\mathcal{O}\neq
\emptyset$ and that $\lambda_1=\inf\{\lambda:\mathcal{O}_{\lambda}\neq
\emptyset\}$ and $\lambda_2=\sup\{\lambda:\mathcal{O}_{\lambda}\neq
\emptyset\}.$ Consider $\epsilon >0$ and the following superset of
$\mathcal{O}$ in $\mathbb{R}\times E,$
\be\label{EqO_eps}
\mathcal{O}^{\epsilon}:=\mathcal{O} \bigcup \Big((\lambda_1-\epsilon,\lambda_1)\times \mathcal{O}_{\lambda_1} \cup  (\lambda_2,\lambda_2+\epsilon)\times \mathcal{O}_{\lambda_2} \Big).
\ee
Then $\mathcal{O}^{\epsilon}$ is an open bounded subset of
$\mathbb{R}\times E.$  Since  $\overline{\mathcal{O}}$ is closed by
definition and $G$ is continuous, then according to the Dugundgi extension
theorem on metric spaces \cite[Thm.~6.1 p.~188]{Dugundji} (cf.~Lemma \ref{Lem_Dugun} below), $G$ can be
extended to $\mathbb{R}\times E$ as a continuous function that we
denote by $\widetilde{G}$.

Now consider,
$$\forall\; (\lambda,u)\in \mathbb{R}\times E,\; H(\lambda,u):=(\lambda -\lambda^{\ast}; u-\widetilde{G}(\lambda, u)), $$
with some arbitrary fixed $\lambda^{\ast}\in [\lambda_1,\lambda_2]$. Then
$H$ is a completely continuous perturbation of the
identity\footnote{This statement can be proved by relying on
the construction of the continuous extension used in the proof of
the Dugundgi theorem. For the sake of completeness, we sketch the 
proof of the latter  in Appendix \ref{Sec_AppendixB}; see Lemma \ref{Lem_Dugun}.} in $\mathbb{R}\times
E$. {\mkrevb In what follows, one denotes by $\widetilde{E}$  the set $\mathbb{R}\times E$.}

Since $H(\lambda,u)=0_{\widetilde{E}}$ if and only if
$\lambda=\lambda^{\ast}$ and $u=\widetilde{G}(\lambda,u),$ and since 
$\lambda^{\ast} \in[\lambda_1,\lambda_2]$ and $G(\lambda, u) \neq u$ on $\partial
\mathcal{O}$ by assumptions, we deduce that,
\begin{equation}\label{Eq_f-eps-nonzero}
\forall\; (\lambda,u)\in \partial \mathcal{O}^{\epsilon},\;
H(\lambda,u)\neq 0_{\widetilde{E}}.
\end{equation}
Therefore $\deg (H,
\mathcal{O}^{\epsilon},0_{\widetilde{E}})$ is well defined and
constant.

Let us consider  the following one-parameter
family $\{H_t\}_{t\in [0,1]}$ of perturbations of
$H$ defined by,
$$\forall\; (\lambda,u)\in \mathbb{R}\times E,\;H_t(\lambda,u):=(\lambda -\lambda^{\ast}; u-t \widetilde{G}(\lambda,u)-(1-t)\widetilde{G}
(\lambda^*,u)).$$
Then
\be
\Big(H_t(\lambda,u)=0\Big) \Leftrightarrow \Big(\lambda=\lambda^{\ast} \mbox{ and } u=\widetilde{G}(\lambda^*,u)\Big),
\ee
and from our assumptions, we conclude again that
$H_t(\lambda,u)\neq 0_{\widetilde{E}}$ for all $(\lambda,u)
\in
\partial \mathcal{O}^{\epsilon}$ and all $t\in [0,1]$.

By applying now the standard homotopy invariance principle
to the family $\{H_t\}_{t\in [0,1]}$ we have
\begin{equation}\label{Eq_homotopy-resI}
\deg(H_1, \mathcal{O}^{\epsilon},0_{\widetilde{E}})=\deg
(H, \mathcal{O}^{\epsilon},0_{\widetilde{E}})= \deg
(H_0,\mathcal{O}^{\epsilon},0_{\widetilde{E}}).
\end{equation}

Let $K$ be the closed subset of $\overline{\mathcal{O}^{\epsilon}}$
such  that $\mathcal{O}^{\epsilon}\backslash
K=(\lambda_1-\epsilon,\lambda_2+\epsilon)\times \mathcal{O}_{\lambda^{\ast}}.$ Then
 $0_{\widetilde{E}}$ does not belong to $H(\partial
\mathcal{O}^{\epsilon} \cup K)$ since the cancelation of $H$
is possible only on the $\lambda^{\ast}$-cross section, while $K$ does
not intercept this section by construction and
$0_{\widetilde{E}}\not \in H(\partial
\mathcal{O}^{\epsilon})$ from \eqref{Eq_f-eps-nonzero}. By applying
now the excision property of the Leray-Schauder degree \cite{Deimling,nir}
with such a $K$, we obtain,
\begin{equation}\label{Eq_var-to-cartesian}
\deg (H_0,\mathcal{O}^{\epsilon},0_{\widetilde{E}})=\deg
(H_0,(\lambda_1-\epsilon, \lambda_2+\epsilon)\times
\mathcal{O}_{\lambda^{\ast}}, 0_{\widetilde{E}}).
\end{equation}

The interest of \eqref{Eq_var-to-cartesian} relies on the fact that
the degree is by this way expressed on a cartesian product which
allows us to apply the cartesian product formula (see Lemma
\ref{Lem_carte-pdct}) and gives in our case
\begin{equation}\label{Eq_appli-carte-pdctI}
\deg (H_0,(\lambda_1-\epsilon, \lambda_2+\epsilon)\times
\mathcal{O}_{\lambda^{\ast}},0_{\widetilde{E}})=\deg(I-G(\lambda^{\ast},\cdot),
\mathcal{O}_{\lambda^{\ast}},0_E),\end{equation} since
$\deg(f,(\lambda_1-\epsilon, \lambda_2+\epsilon), 0_{\mathbb{R}})=1$ with
$f(\lambda)=\lambda -\lambda^*,$ and $\lambda^* \in [\lambda_1,\lambda_2]$.

By applying now \eqref{Eq_appli-carte-pdctI},
\eqref{Eq_var-to-cartesian} and \eqref{Eq_homotopy-resI} and by 
recalling that $\deg (H,
\mathcal{O}^{\epsilon},0_{\widetilde{E}})$ is independent of
$\lambda^*$, we have thus proved that for arbitrary $\lambda^* \in
[\lambda_1,\lambda_2],$ $\deg(I-G(\lambda^{\ast},\cdot),
\mathcal{O}_{\lambda^{\ast}},0_E)$ is also independent of
$\lambda^*.$ The proof is complete.
\end{proof}

\begin{Rmk}\label{Rem_epsilon}
The introduction of $\mathcal{O}^{\epsilon}$ such as defined in \eqref{EqO_eps}
above was used in order to work within an open bounded subset of a Banach space,
here $\mathbb{R}\times E$, and thus to work within the framework of
the Leray-Schauder degree\footnote{the original open subset $\mathcal{O}$ is not an open subset of a
Banach space, but of the (complete) metric space $[\lambda_1,\lambda_2]\times E$.}. The Dugundgi theorem is used to appropriately extend
the mapping $G$ to $\mathcal{O}^{\epsilon}$ in order to apply the Leray-Schauder
degree techniques.
\end{Rmk}

The last ingredient to prove Theorem \ref{THM_global_unbounded}, is the following 
separation lemma from point set topology (Lemma \ref{Lem_separation}
below). A {\it separation} of a topological space $X$ is a pair of
nonempty open subsets $U$ and $V$, such that $U\cap V=\emptyset$ and
$U\cup V=X$. A space is connected  if it does not admit a
separation. Two subsets $A$ and $B$ are connected in $X$ if the
exists a connected set $Y\subset X,$ such that $A\cap Y\neq
\emptyset$ and $B\cap Y \neq \emptyset$. Two nonempty subsets $A$
and $B$ of $X$ are {\it separated} if there exists a separation
$U,V$ of $X$ such that $A\subseteq U$ and $B\subseteq V.$ There
exists a relationship between these concepts in the case where $X$
compact, this is summarized in the following separation lemma.

\begin{Lem}\label{Lem_separation}
(Separation lemma) If $X$ is compact and $A$ and $B$ are not
separated, then $A$ and $B$ are connected in $X$.
\end{Lem}
The proof of this  lemma may be found in \cite[Lemma 29.1]{Deimling}; see also
\cite{Kuratowski}.

 As a result if two subsets of
a compact set are not connected, they are separated. We are now in
position to prove Theorem \ref{THM_global_unbounded}.

\vspace{1ex}
\medskip

\noindent{\bf Proof of Theorem \ref{THM_global_unbounded}.}
\vspace{-1ex}
\begin{proof}
Let $\mathcal{C}^{+}$ be the maximal connected subset of $\mathcal{S}^{+}$
such that (i) holds, which is trivial by assumptions. We proceed by
contradiction. Assume that $\mathcal{C}^{+}_{\lambda_0}\cap(E\setminus
\overline{\mathcal{U}})= \emptyset $ and that $\mathcal{C}^{+}$ is bounded in
$[\lambda_0,\infty)\times E$. Then there exists a constant $R>0$
such that for each $(\lambda,u)\in \mathcal{C}^{+}$ we have $\|u\|+|\lambda|
<R$. Introduce,
$$\mathcal{S}^{+}_{2R}:=\{(\lambda,u)\in \mathcal{S}^+\;:\;\|u\|+|\lambda| \leq 2R \}. $$
From the complete continuity of $G$ it follows that any set of the
form $\mathcal{H}:=\{(\lambda,u)\in \Lambda\times E\;:\;
u=G(\lambda,u)\},$ with $\Lambda$ a closed and bounded subset of
$[\lambda_0, \infty),$ is a compact subset of $[\lambda_0, \infty)
\times E$. As a result, $\mathcal{S}^{+}_{2R}$ is a compact subset
of $[\lambda_0, \infty) \times E$.

There are two possibilities. Either (a) $\mathcal{S}^{+}_{2R}=\mathcal{C}^{+}$
or, (b) there exists $(\lambda^*,u^*)\in \mathcal{S}^{+}_{2R}$ such
that $(\lambda^*,u^*)$  does not belong to $\mathcal{C}^+$.

Let $\mathcal{U}$ be as defined in Theorem \ref{THM_global_unbounded}.
Consider case (b) first. We want to apply Lemma \ref{Lem_separation}
with $X=\mathcal{S}^{+}_{2R},$ $A=\mathcal{C}^+,$ and $B= \{\lambda^*\}\times
\mathcal{S}_{2R}^+$. Obviously, $A$ and $B$ are not connected in
$\mathcal{S}^{+}_{2R}$ since $(\lambda^*,u^*)\not\in \mathcal{C}^+$ and $\mathcal{C}^+$
is the maximal connected subset of $\mathcal{S}^+$. We may therefore
apply Lemma \ref{Lem_separation} in such a case and build an open
subset $\mathcal{O}$ of $[\lambda_0, \infty) \times E$, such that
the following properties hold,
\begin{itemize}
\item[(c$_1$)] $\mathcal{O}_{\lambda_0}=\mathcal{U}$ (since
$\mathcal{C}^{+}_{\lambda_0}\cap(E\setminus \overline{\mathcal{U}})= \emptyset
$),
\item[(c$_2$)]$\mathcal{C}^{+} \subset \mathcal{O}$,
\item[(c$_3$)]$\mathcal{S}^{+}_{2R}\cap
\partial \mathcal{O}=\emptyset$ and,
 \item[(c$_4$)]
 $\mathcal{O}_{\lambda^*}$
contains no solutions of $u=G(\lambda^*,u)$.
\end{itemize}
The last property comes from the fact that $A$ and $B$, as defined
above, are separated.

From (c$_3$), we get by applying Lemma \ref{Lem_Homotopy_gene},
that,
\begin{equation}\label{Eq_Homotopy-conseq-II}
\forall\; \lambda \in \Lambda_R,\; \deg(I-G(\lambda,\cdot),
\mathcal{O}_{\lambda},0)=\deg(I-G(\lambda_0,\cdot),
\mathcal{O}_{\lambda_0},0),
\end{equation}
where $\Lambda_R$ denotes the projection of $\mathcal{S}^{+}_{2R}$
onto $[\lambda_0,\infty)$.

Now $\deg(I-G(\lambda_0,\cdot), \mathcal{O}_{\lambda_0},0)\neq 0$ by
(c$_1$) and the assumptions of Theorem \ref{THM_global_unbounded}.  We
obtain therefore a contradiction from (c$_4$) when
\eqref{Eq_Homotopy-conseq-II} is applied for $\lambda=\lambda^*$ .

The case $\mathcal{C}^+=\mathcal{S}^{+}_{2R},$  may be treated along the same
lines and is left to the reader. The proof is complete.
%$\hspace{14cm}
\end{proof}

\begin{Rmk}\label{Rem-unique-unbounded}
Theorem \ref{THM_global_unbounded} shows in particular that if for all $\mathcal{U}$ there is
  a unique solution $(\lambda_0,u_0)$ in $\mathcal{U}$, of
$u=G(\lambda_0,u)$, then there exists an
{\it unbounded} continuum  of solutions of $u=G(\lambda,u),$
provided that there exists an open set $\mathcal{V}$ in $E$ such
that $\deg(I-G(\lambda_0, \cdot),\mathcal{V},0)\neq 0$.
\end{Rmk}

\begin{Rmk}
It is not essential that $u_0$ be the only solution of
\eqref{Eq_Nonlinear-eq} in $\mathcal{U}$.  If one only assumes
\eqref{Eq_deg-non-zero}, one obtains  the
existence of finitely many continua satisfying the
alternative formulated in (ii) of Theorem \ref{THM_global_unbounded}.
\end{Rmk}

\section{Product formula for the Leray-Schauder degree, and the Dugundji extension theorem}\label{Sec_AppendixB}
This appendix contains auxiliaries lemmas used in the previous Appendix. We first start with the cartesian product formula for the
Leray-Schauder degree.

\begin{lemma}\label{Lem_carte-pdct}
Assume that $\mathcal{U}=\mathcal{U}_1 \times \mathcal{U}_2$ is a bounded open
subset of $E_1\times E_2$, where $E_1$ and $E_2$ are two real Banach
spaces with $\mathcal{U}_1$ and $\mathcal{U}_2$ open subsets of $E_1$ and
$E_2$ respectively. Suppose that for all $x=(x_1,x_2) \in E$,
$f(x)=(f_1(x_1),f_2(x_2)),$ where
$f_1:\overline{\mathcal{U}_1}\rightarrow E_1$ and
$f_2:\overline{\mathcal{U}_2}\rightarrow E_2$ are continuous and suppose
that $y=(y_1,y_2) \in E$ is such that $y_1$ (resp. $y_2$) does not
belong to $f_1(\partial \mathcal{U}_1)$ (resp.  $f_2(\partial \mathcal{U}_2)$).
Then,
$$ \deg(f,\mathcal{U},y)=\deg(f_1,\mathcal{U}_1,y_1)\deg(f_2,\mathcal{U}_2,y_2).$$
\end{lemma}

We recall below the {\it Dugundgi extension theorem} \cite[Thm.~6.1 p.~188]{Dugundji}.
\begin{lemma}\label{Lem_Dugun}
({\bf Dugundgi}) Let $E$ and $X$ be Banach spaces and let $f:\mathfrak{C}\rightarrow E$ a
continuous mapping, where $\mathfrak{C}$ is a closed subset of $E$.
%, and $K$ is a convex subset of $X$. 
Then there exists a continuous mapping $\widetilde{f}:E\rightarrow
K$ such that $\widetilde{f}(u)=f(u)$ for all $u\in \mathfrak{C}$.
\end{lemma}

\begin{proof}(Sketch)
For each $u\in E\backslash \mathfrak{C},$ let $r_u=\frac{1}{3}\mbox{dist}(u,\mathfrak{C}),
$ and $B_u:=\{v\in E\; : \; \|v-u\|<r_u\}$. Then
$\mbox{diam}(B_u)\leq \mbox{dist}(B_u, \mathfrak{C}),$ and $\{B_u\}_{u\in
E\backslash \mathfrak{C}}$ is a open cover of $E\backslash \mathfrak{C}$ which admits a
local refinement $\{\mathcal{O}_{\lambda}\}_{\lambda \in \Lambda}$:
i.e.~$\underset{\lambda\in \Lambda}\bigcup
\mathcal{O}_{\lambda} \supset E\backslash \mathfrak{C}$, for each $\lambda \in
\Lambda$ there exists $B_u$ such that $B_u\supset \mathcal{O}_{\lambda}$,
and every $u\in E\backslash  \mathfrak{C}$ has a neighborhood $U$ such that  ~$U$
intersects at most finitely many elements of
$\{\mathcal{O}_{\lambda}\}_{\lambda\in \Lambda}$ (locally finite
family).

Introduce now $\gamma: E\backslash \mathfrak{C} \rightarrow \mathbb{R}^+_*$, defined
by $\gamma(u)=\underset{\lambda\in \Lambda}\sum
\mbox{dist}(u,\overline{\mathcal{O}_{\lambda}})$ and introduce
$$\forall \lambda \in \Lambda,\; \forall u\in E\backslash \mathfrak{C},\;\; \gamma_{\lambda}(u)=\frac{\mbox{dist}(u,\overline{\mathcal{O}_{\lambda}})}{\gamma(u)}.$$
By construction, the above sum over $\Lambda$ contains only finitely
many terms and thus $\gamma$ is continuous.

Now define $\widetilde{f}$ by,

\begin{equation}\label{LGeq}
\widetilde{f}=\left\{
\begin{array}{l}
f(u), \mbox{ if } u\in \mathfrak{C},\\
\sum_{\lambda \in \Lambda} \gamma_{\lambda}(u) f(u_{\lambda}), \;
u\not\in \mathfrak{C}.
\end{array}\right.
\end{equation}
Then it can be shown that $\widetilde{f}$ is continuous.

\end{proof}

%%%%%%%%%%%%%
\section*{Acknowledgments}
The author is grateful to Thierry Cazenave for  the stimulating discussions concerning the reference \cite{Nao} at the start of this project.  The author thanks also Lionel Roques, Jean Roux  and Eric Simonnet for their interests in this work, and Honghu Liu for his help in preparing Figure 1.  This work was partially supported 
by the grant N00014-16-1-2073 from the Multidisciplinary University Research Initiative (MURI) of the Office of Naval Research, and by the National Science Foundation  grants OCE-1658357 and DMS-1616981.

\bibliographystyle{amsalpha}
%\bibliography{Gelfand}

\end{document}